\documentclass[11pt,reqno]{amsart}
\usepackage{a4wide}
\usepackage{empheq}

\usepackage{amscd,mathtools,amsthm}
\usepackage[svgnames]{xcolor}
\usepackage{amsfonts}

\usepackage{amssymb}
\usepackage{longtable}
\usepackage[all]{xy}
\usepackage{graphicx}
\usepackage{enumerate}
\usepackage{tcolorbox}
\usepackage{fullpage}
\usepackage{titletoc}
\usepackage[colorlinks=true, urlcolor=blue, linkcolor=blue, citecolor=black]{hyperref}
\usepackage[nameinlink]{cleveref}

\usepackage[utf8]{inputenc}

\usepackage{esint}
\usepackage{verbatim}
\usepackage{mathrsfs}
\usepackage{cancel}
\usepackage{tikz}
\usetikzlibrary{cd}

\newcommand{\ol}{\overline}
\newcommand{\3}{\varepsilon}
\newcommand{\4}{\widetilde}

\newcommand{\R}{{\mathbb R}}

\newcommand{\round}[1][t]{\frac{\partial}{\partial {#1}}}

\theoremstyle{plain}
\newtheorem{theorem}{Theorem}[section]
\newtheorem{lemma}[theorem]{Lemma}
\newtheorem{corollary}[theorem]{Corollary}

\theoremstyle{definition}

\newtheorem{remark}[theorem]{Remark}

\numberwithin{equation}{section}

\crefformat{equation}{#2(#1)#3}
\Crefformat{equation}{\textup{#2(#1)#3}}

\def\XXint#1#2#3{{\setbox0=\hbox{$#1{#2#3}{%
        \int}$ }
    \vcenter{\hbox{$#2#3$ }}\kern-.6\wd0}}

\makeatletter
\@namedef{subjclassname@2020}{\textup{2020} Mathematics Subject Classification}
\makeatother

\title[Asymptotic large time behavior of singular entire solutions of the fast diffusion equation]{Asymptotic large time behavior of singular entire solutions of the fast diffusion equation  }

\author{Kin Ming Hui}
\address{Institute of Mathematics, Academia Sinica, Taipei 106319, Taiwan}
\email{kmhui@gate.sinica.edu.tw}

\author{Jongmyeong Kim}
\address{Institute of Mathematics, Academia Sinica, Taipei 106319, Taiwan}
\email{jmkim@gate.sinica.edu.tw}

\subjclass[2020]{35B35, 35B44, 35J70, 35K55, 35K65}
\keywords{fast diffusion equation; singular solution; singular self-similar solutions; large time behaviour}

\begin{document}

\begin{abstract}
Let $n\ge 3$, $0<m<\frac{n-2}{n}$, $\alpha=\frac{2\beta-1}{1-m}$ and $\frac{2}{1-m}<\frac{\alpha}{\beta}<\frac{n-2}{m}$. We give a new direct proof using fixed point method on the existence  of singular radially symmetric forward self-similar solution of the form $V(x,t)=t^{-\alpha} f(t^{-\beta}x)$ $\forall x\in\mathbb{R}^n\setminus\{0\}$, $t>0$, for the fast diffusion equation $u_t=\Delta (u^m/m)$  in $(\mathbb{R}^n\setminus\{0\})\times (0,\infty)$, where 
$f$ satisfies 
\begin{equation*}
\Delta (f^m/m) + \alpha f + \beta x \cdot \nabla f =0 \quad \text{in} \;  \mathbb{R}^n\setminus\{0\}
\end{equation*}
with $\lim_{|x| \to 0} |x|^{ \frac{\alpha}{\beta}}f(x)=A$ and $\lim_{|x| \to \infty}f(x) = D_A$ for some constants $A>0$, $D_A > 0$.  We also obtain an asymptotic expansion of such singular radially symmetric solution $f$ near the origin. We will also prove the asymptotic large time behaviour of the singular solutions of the fast diffusion equation 
$u_t= \Delta (u^m/m)$ in $(\mathbb{R}^n\setminus\{0\})\times (0,\infty)$, $u(x,0)=u_0(x)$ in $\mathbb{R}^n\setminus\{0\}$,  satisfying the  condition
$A_1|x|^{-\gamma}\leq u_0(x)\leq A_2|x|^{-\gamma}$ in $\mathbb{R}^n\setminus\{0\}$, 
for some constants $A_2>A_1>0$ and $n\le\gamma<\frac{n-2}{m}$. 
\end{abstract}

\date{April 26, 2025}

\maketitle


\section{Introduction} 
The equation 
\begin{equation}\label{fde}
u_t= \Delta (u^m/m)
\end{equation}
arises from many physical phenomena. When $m=1$, it is the heat equation. When $m>1$, the equation is called the porous medium equation which arises from diffusion process in porous medium  such as oil passing through sand \cite{Aro06porous} and  as the large time asymptotic limit in the study of the large time behaviour of the solution of the compressible Euler equation with damping \cite{LZ16global}. Porous medium type equation also appears in the study of the Hele-Shaw flow with general initial density \cite{KP18porous}.
When $0<m<1$, it is called the fast diffusion equation. In this case  the equation appears in the modelling of plasma physics \cite{BH82asympt} and diffusion of impurities in silicon \cite{Kin88extrem}, \cite{Kin93self}. 

When $n \ge 3$ and $m=\frac{n-2}{n+2}$, the equation \Cref{fde} appears in the study of Yamabe flow. We say that a metric $g$ on a Riemannian manifold $M$ evolves by  the Yamabe flow on $(0,T)$ \cite{Bre05conver}, \cite{Bre07conver}, \cite{dPS01extinc}, if it satisfies
\begin{equation}\label{eq Yamabe flow}
\round{g} = -Rg 
\end{equation}
on $M$ for any $0<t<T$ where $R$ is the scalar curvature of the metric at time $t$. When the manifold $M$ is conformal to $\mathbb{R}^n$, we can write the metric as $g=u^{\frac{2}{n+2}}dx^2$. Then 
the scalar curvature is given by 
\begin{equation}\label{scalar-curvature}
R=-\frac{4(n-1)}{(n-2)}u^{-1}\Delta u^{\frac{n-2}{n+2}}.
\end{equation}
By \eqref{eq Yamabe flow} and \eqref{scalar-curvature}, $u$ satisfies
\begin{equation*}
u_t=\frac{n-1}{m}\Delta u^m\quad\mbox{ in }\mathbb{R}^n\times (0,T)
\end{equation*}
with $m=\frac{n-2}{n+2}$ which is equivalent to \eqref{fde} with $m=\frac{n-2}{n+2}$ after a rescaling. 
Hence the study of properties of conformal Yamabe flow on $\mathbb{R}^n$ is equivalent to the study of the the properties of the solutions of the equation  \eqref{fde} in the domain $\mathbb{R}^n$ with $m=\frac{n-2}{n+2}$. We refer the reader to the papers by P.~Daskalopoulos, M.~del Pino, J.~King and N.~Sesum, \cite{DDPKS16type}, \cite{DDPKS17new}, S.Y.~Hsu \cite{Hsu12singul}, \cite{Hsu17existe}, M.~del Pino and M.~S\'aez \cite{dPS01extinc}, J.~Takahashi and H.~Yamamoto \cite{TY22infini}, etc. for some of the recent research in this direction. We also refer the readers to the book  \cite{Vaz06smooth} by J.L.~Vazquez  and the book \cite{DK07degene} by P.~Daskalopoulos and Carlos E.~Kenig for some recent results on \eqref{fde}.

As observed by M.A.~Herrero, M.~Pierre and J.L.~Vazquez and others \cite{HP85cauchy}, \cite{Vaz92nonexi}, there are big differences in the behaviour of the solutions of \eqref{fde} for the cases $m>1$, $m_c:=\frac{n-2}{n}<m <1$ and $0<m<\frac{n-2}{n}$. When $m>1$, the solution $u$ of 
\begin{equation}\label{eq fast diff cauchy problem}
\left\{\begin{aligned}
&u_t= \Delta (u^m/m)\quad\mbox{ in }\mathbb{R}^n\times (0,T)\\
&u(x,0)=u_0(x)\quad\mbox{ in }\mathbb{R}^n
\end{aligned}\right.
\end{equation}
with compact bounded initial data $u_0\ge 0$ has $\mbox{supp}\, u(\cdot,t)$ being compact for all $t\in (0,T)$. On the other hand for the supercritical regime $m_c:=\frac{n-2}{n} <m <1$, it is known that for any non-zero non-negative initial data $u_0$ there exists a unique global solution of \eqref{eq fast diff cauchy problem} which is positive for all time \cite{HP85cauchy}. 

When $0<m<m_c$ and $T>0$, there exist solutions of \eqref{fde} in $\mathbb{R}^n\times (0,T)$ which is positive and smooth in $\mathbb{R}^n\times (0,T)$ and vanishes in a finite time $T>0$. For example
the Barenblatt solution \cite{DS08extinc},
\begin{equation*}
B_k(x,t)=(T-t)^{\alpha_1}\left(\frac{C_{\ast}}{k^2+|(T-t)^{\beta_1}x|^2}\right)^{\frac{1}{1-m}}.
\end{equation*}
where $C_{\ast}=\frac{2(n-2-nm)}{1-m}$, $k>0$, $T>0$, 
\begin{equation*}
\beta_1=\frac{1}{n-2-nm}\quad\mbox{ and }\quad\alpha_1=\frac{2\beta_1+1}{1-m},
\end{equation*}
is a self-similar solution of \eqref{fde} in $\mathbb{R}^n\times (0,T)$ which vanishes identically at time $T$.
We refer the readers to the papers by P. Daskalopolous, J. King, M. Sesum, M. S\'aez  \cite{DS08extinc}, \cite{DKS19extinc},   \cite{dPS01extinc} S.Y. Hsu \cite{Hsu12singul,Hsu13existe}, K.M. Hui \cite{Hui02some,Hui08singul,Hui17asympt}, M. Fila, J.L. Vazquez, M. Winkler, E. Yanagida \cite{FVWY12rate,FW16optima,FW16rate,FW17slow}, etc.  for some recent results on \eqref{fde} when $0<m<m_c$.

In this paper we will study the asymptotic large time behaviour of the singular solutions of  \eqref{fde} in $(\mathbb{R}^n\setminus\{0\})\times (0,\infty)$ which blow up at the origin for all time $t>0$.
Since the behaviour of such singular solutions of \eqref{fde} is usually similar to the behaviour of the singular forward self-similar solutions of \eqref{fde} in $(\mathbb{R}^n\setminus\{0\})\times (0,\infty)$ which blow up at the origin. Hence in order to understand the behaviour of the solutions of \eqref{fde} it is important to first study the singular forward self-similar solutions of \eqref{fde} in $(\mathbb{R}^n\setminus\{0\})\times (0,\infty)$. Let
\begin{equation}\label{n-m-relation}
n\ge 3\quad\mbox{ and }\quad 0<m<\frac{n-2}{n}.
\end{equation}
Suppose $V$ is a singular radially symmetric forward self-similar solution of  \eqref{fde} in $(\mathbb{R}^n\setminus\{0\})\times (0,\infty)$ of the form
\begin{equation}\label{eq:one point singular:24}
V(x,t)=t^{-\alpha} f(t^{-\beta}x)\quad\mbox{ in  }(\mathbb{R}^n\setminus\{0\})\times (0,\infty)
\end{equation} 
for some constants $\alpha<0$, $\beta<0$, which have initial values of the form $\eta |x|^{-\gamma}$ in $\mathbb{R}^n\setminus\{0\}$ for some constants $\eta>0$ and 
\begin{equation}\label{eq:gamma range}
\frac{2}{1-m}<\gamma<\frac{n-2}{m}. 
\end{equation}
Then $V(x,t)$ is a solution of 
\begin{equation}\label{eq-fde-global-except-0} 
\left\{\begin{aligned}
u_t=&\Delta (u^m/m),\quad u>0,\quad\mbox{in }(\R^n\setminus\{0\})\times(0,\infty)\\
u(\cdot,0)=&\eta |x|^{-\gamma}\qquad\qquad\qquad\mbox{ in }\R^n\setminus\{0\}\end{aligned}\right.
\end{equation}  
if and only if $f$ is a radially symmetric solution of
\begin{equation}\label{eq:one point singular:25}
  \Delta (f^m/m) + \alpha f + \beta x \cdot \nabla f =0,\quad f>0, \quad \text{in} \;  \mathbb{R}^n \setminus   \left\{ 0 \right\}
\end{equation}
satisfying
\begin{equation}\label{f-blow-up-rate-origin}
 \lim_{|x| \to 0 } |x|^{\gamma} f(x) = \eta
\end{equation}
with
\begin{equation}\label{eq:alpha beta relation}
\alpha=\frac{2\beta-1}{1-m}
\end{equation} 
and
\begin{equation}\label{eq:gamma slpha beta relation6}
\gamma=\alpha/\beta.
\end{equation}
Note that by the result of \cite{Hsu25non} if there exists a radially symmetric solution $f$ of \eqref{eq:one point singular:25}, \eqref{f-blow-up-rate-origin}, for some constants $\alpha\in\mathbb{R}$,  $\beta\ne 0$, $\gamma>0$ and $\eta>0$, satisfying \eqref{eq:alpha beta relation}, then \eqref{eq:gamma slpha beta relation6}
holds and $\gamma>\frac{2}{1-m}$.
By \eqref{eq:alpha beta relation} and \eqref{eq:gamma slpha beta relation6},
\begin{equation}\label{eq:alpha beta relation5}
\beta=\frac{1}{2-\gamma (1-m)}<0\quad\mbox{ and }\alpha=\frac{2\beta-1}{1-m}<0
\end{equation}
since $ \gamma >\frac{2}{1-m}$.
Then 
\begin{equation}\label{alpha'-beta'-defn}
\alpha':=-\alpha >0, \; \beta' :=-\beta>0.
\end{equation}
Note that when $\gamma$ satisfies \eqref{eq:gamma range},
\begin{equation}\label{eq:beta0}
\beta < \beta^{(m)}_0 := \frac{-m}{n-2-mn}.
\end{equation}

In \cite{HK17asympt} K.M.~Hui and Soojung Kim by using an inversion method together with a fixed point argument on the equation obtained by the inversion transform proved the existence and uniqueness of singular radially symmetric  solution of \eqref{eq:one point singular:25},  \eqref{eq:one point singular:27}, when \eqref{n-m-relation} holds and  $\alpha$, $\beta$, $\gamma$ satisfy \eqref{eq:gamma range} and \eqref{eq:alpha beta relation}, \eqref{eq:gamma slpha beta relation6}. However this method may not be applicable to find the other singular self-similar solutions of \eqref{fde}.

In this paper we will give a new direct proof of this existence and uniqueness result of the singular self-similar solutions of K.M.~Hui and Soojung Kim \cite{HK17asympt}  using fixed point method. Similar fixed point approach has been used by K.M.~Hui \cite{Hui23existe} to prove the
existence and uniqueness of backward singular radially symmetric solution of \eqref{fde} in 
$(\mathbb{R}^n\setminus\{0\})\times (0,T)$, $T>0$, which blows up on $\{0\}\times (0,T)$. In \cite{Hui23existe} K.M.~Hui also used this method to prove the existence and uniqueness of backward singular radially symmetric solution of the logarithmic diffusion equation in 
$(\mathbb{R}^n\setminus\{0\})\times (0,T)$ which blows up on $\{0\}\times (0,T)$.

In \cite{HK17asympt} K.M.~Hui and Soojung Kim also proved the asymptotic behaviour of the solution of
\begin{equation}\label{eq-fde2-global-except-0} 
\left\{\begin{aligned}
u_t=\Delta u^m\quad&\mbox{in }(\mathbb{R}^n\setminus\{0\})\times(0,\infty)\\
u(\cdot,0)=u_0(x)\quad&\mbox{in }\mathbb{R}^n\setminus\{0\}\end{aligned}\right.
\end{equation}
with the initial value $u_0$ satisfying the growth condition
\begin{equation}\label{eq-fde2-initial}
A_1|x|^{-\gamma}\leq u_0(x)\leq A_2|x|^{-\gamma}\quad\mbox{in $\mathbb{R}^n\setminus\{0\}$} 
\end{equation}
for some constants $A_2>A_1>0$ and $\frac{2}{1-m}< \gamma <n$. In this paper we will extend this asymptotic behaviour result to the case
\begin{equation}\label{gamma-range1}
n\le  \gamma < \frac{n-2}{m}.
\end{equation}
We will show that in this case under a mild addition assumption on the initial data some rescaled function of the solution of \eqref{eq-fde2-global-except-0}, \eqref{eq-fde2-initial}, will converge to some forward singular radially symmetric self-similar solution of the equation. 

Finally we will use a modification of the technique of K.M.~Hui and Sunghoon Kim \cite{HK19singul} to  prove the  asymptotic expansion of such forward singular radially symmetric solution $f$ of \eqref{eq:one point singular:25}, \eqref{eq:one point singular:27}, near the origin. 

More precisely we will prove the following main results. In particular we will give a new proof of the following theorem.

\begin{theorem}[Existence and uniqueness of singular self-similar solution, cf. Theorem 1.1 of \cite{HK17asympt}]\label{existence-self-similar-soln-thm}
Let $n \ge 3, \; 0<m<\frac{n-2}{n}, \;  \eta>0$ and
\begin{equation}\label{eq:alpha beta relation3}
\beta<0,\quad \rho_1>0,\quad \alpha=\frac{2\beta-\rho_1}{1-m}\quad\mbox{ and }\quad \frac{2}{1-m}<\frac{\alpha}{\beta}<\frac{n-2}{m}. 
\end{equation}
Then there exists a unique radially symmetric singular solution $f$  of \eqref{eq:one point singular:25} satisfying 
\begin{equation}\label{eq:one point singular:27}
 \lim_{|x| \to 0 } |x|^{\alpha/\beta} f(x) = \eta.
\end{equation} 
Moreover $f$ satisfies
\begin{equation}\label{f-infty-behaviour}
\lim_{|x| \to \infty}|x|^{\frac{n-2}{m}}f(x) = D(\eta)
\end{equation}
for some constant $D(\eta) > 0$.
\end{theorem}

The following theorem extends the result of K.M.~Hui and Sunghoon Kim \cite{HK19singul} on the asymptotic expansion of singular backward radially symmetric self-similar solution of \eqref{fde} near the origin to the case of singular forward radially symmetric self-similar solution of \eqref{fde} near the origin.

\begin{theorem}[Asymptotic expansion of singular radially symmetric self-similar solution near the origin]\label{asymptotic expansion of self-similar solution at the origin thm}
Let $n \ge 3$, $0<m<\frac{n-2}{n}$, $\eta>0$ and let $\alpha$, $\beta$, $\rho_1$, $\alpha'$, 
$\beta'$ satisfy \eqref{alpha'-beta'-defn} and \eqref{eq:alpha beta relation3}. Let $f$ be a radially symmetric solution of \Cref{eq:one point singular:25}  which satisfies \Cref{eq:one point singular:27}. Let 
\begin{equation}\label{w-defn}
w(r) = r^{\frac{\alpha}{\beta} }f(r),
\end{equation}
 $\rho=r^{\rho_1/\beta'}$ and $\overline{w}(\rho) =w(r)$. Then $\overline{w}$ can be extended to a function in $C^2 ([0,\infty) )  $ by setting
\begin{equation}\label{eq:one point singular:28} 
\left\{\begin{aligned}
 & \overline{w}(0)=\eta\\
 & \overline{w}_{\rho}(0)=\frac{a_3}{a_2}\eta^m\\
   & \overline{w}_{\rho \rho}(0) = \frac{a_3(ma_3-a_1)}{a_2^2} \eta^{2m-1} 
\end{aligned}\right.
\end{equation}
where
\begin{equation}\label{eq:one point singular:30}
 a_1 = \frac{2m\alpha-(n-2) \beta+\rho_1}{ \rho_1}  ,\; a_2= -\frac{\beta^2}{ \rho_1}\quad{ and } \; a_3 = \frac{ \alpha \beta(n-2) - m \alpha^2}{\rho_1^{2}} .
\end{equation}
Hence 
\begin{align*}
\left\{\begin{aligned}
&\ol{w}(\rho)=\eta+\frac{a_3}{a_2}\eta^m\rho+\frac{a_3(ma_3-a_1)}{2a_2^2} \eta^{2m-1}\rho^2+o(\rho^2)\quad\mbox{ as }\rho\to 0^+\\
&\ol{w}_{\rho}(\rho)=\frac{a_3}{a_2}\eta^m\rho+\frac{a_3(ma_3-a_1)}{a_2^2} \eta^{2m-1}\rho^2+o(\rho)\quad\mbox{ as }\rho\to 0^+
\end{aligned}\right.
\end{align*}
or equivalently

\begin{align*}
\left\{\begin{aligned}
&f(r)=r^{-\alpha/\beta}\left\{\eta+\frac{a_3}{a_2}\eta^m r^{\rho_1/\beta'}+\frac{a_3(ma_3-a_1)}{2a_2^2} \eta^{2m-1}r^{2\rho_1/\beta'}+o(r^{2\rho_1/\beta'})\right\}\quad\mbox{ as }r\to 0^+\\
&f_r(r)=r^{-\frac{\alpha}{\beta}-1}\left\{-\frac{\alpha}{\beta}\eta-\frac{(2\beta-m\rho_1)a_3}{(1-m)a_2\beta}\eta^mr^{\rho_1/\beta'}+o(r^{\rho_1/\beta'})\right\}\quad\mbox{ as }r\to 0^+.
\end{aligned}\right.
\end{align*}
\end{theorem}

\begin{remark}
For any $\lambda>0$, we let $f_{\lambda}$ be the solution of \eqref{eq:one point singular:25}, \eqref{eq:one point singular:27}, with $\eta=\lambda^{\frac{1}{ (1-m) \beta}}$  and $\alpha$, $\beta$,  $\gamma$, satisfy \eqref{eq:gamma range} and \eqref{eq:alpha beta relation}, \eqref{eq:gamma slpha beta relation6}, or equivalently \eqref{eq:alpha beta relation5}  given by Theorem \ref{existence-self-similar-soln-thm}. Then
\begin{equation}\label{eq:V lambda}
V_{\lambda}: = t^{-\alpha} f_{\lambda}(t^{-\beta}x)
\end{equation}
is a singular radially symmetric solution of \eqref{eq-fde-global-except-0} with  $\eta= \lambda^{\frac{1}{ (1-m) \beta}}$ which blows up on $\{0\}\times (0,\infty)$. 
Note that by the result of \cite{HK17asympt},
\begin{equation*}
f_{\lambda}(x)=\lambda^{\frac{2}{1-m}}f_1(\lambda x)\quad\forall x\in\mathbb{R}^n\setminus\{0\}
\end{equation*}
and by Remark 2 of \cite{HK17asympt},
\begin{equation*}
\frac{d}{d\lambda}f_{\lambda}(r)<0\quad\forall r>0\quad\Rightarrow\quad f_{\lambda_1}(r)<f_{\lambda_2}(r)\quad\forall\lambda_1>\lambda_2>0, r>0.
\end{equation*}
Moreover by Remark 2 of \cite{HK17asympt} for any $\lambda_1>\lambda_2>0$, there exists a constant $c_0>0$ such that
\begin{equation*}
c_0f_{\lambda_2}(r)<f_{\lambda_1}(r)\quad\forall r>0.
\end{equation*}
\end{remark}

Before stating the convergence theorem we first construct proper weight function used in the theorem.
Let  $0<\mu<n-2$. We will construct a non-negative radially symmetric weight function $\varphi_{\mu}\in C^{\infty} (\mathbb{R}^n)$ as follows. Let $\eta_1$ be a smooth non-negative radially symmetric function on $\mathbb{R}^n$ that satisfies
\begin{equation*}
\eta_1(r) = \begin{cases}
0\quad & \text{if }0\le r=|x| \le 1  \\
\mu((n-2)-\mu) r^{-\mu-2} \quad & \text{if } r \ge 2.
\end{cases}
\end{equation*}
Let $\varphi_1$ be a radially symmetric function given by
\begin{equation} \label{eq:varphi}
\varphi_{\mu}(x)= \varphi_{\mu}(|x|):= 1 -a_4\int_0^{|x|}\frac{1}{s^{n-1}} \left(  \int_0^s \rho^{n-1} \eta_1(\rho) \,\mathrm{d} \rho \right) \,\mathrm{d} s\quad\forall x\in\mathbb{R}^n
\end{equation}
where 
\begin{equation}\label{a4-defn} 
a_4=\left(\int_0^{\infty}\frac{1}{s^{n-1}} \left(  \int_0^s \rho^{n-1} \eta_1(\rho) \,\mathrm{d} \rho \right) \,\mathrm{d} s\right)^{-1}.
\end{equation}  
We will now show that $a_4$ is well-defined. Note that
\begin{align}\label{eta1-integral-value}
\int_0^s \rho^{n-1} \eta_1(\rho) \,\mathrm{d} \rho 
=&\int_1^2 \rho^{n-1} \eta_1(\rho) \,\mathrm{d} \rho + \mu(n-2-\mu) \int_2^s \rho^{n-3-\mu} \,\mathrm{d} \rho \quad\forall s>2\notag\\
=&a_5+ \mu(2^{n-2-\mu} -s^{n-2-\mu} )\quad\forall s>2
\end{align}
where 
\begin{equation*}
a_5=\int_1^2 \rho^{n-1} \eta_1(\rho) \,\mathrm{d} \rho\ge 0. 
\end{equation*}
Hence by \eqref{eta1-integral-value},
\begin{align}\label{integral-estimate1}
\int_0^{\infty}\frac{1}{s^{n-1}} \left(  \int_0^s \rho^{n-1} \eta_1(\rho) \,\mathrm{d} \rho \right) \,\mathrm{d} s
=&\int_1^2 s^{1-n} \left(  \int_1^s \rho^{n-1} \eta_1(\rho) \,\mathrm{d} \rho \right) \,\mathrm{d} s + k_0
\end{align}
where
 \begin{align}\label{k0-defn}
k_0&= \int_2^{\infty}s^{1-n} \left[a_5 +\mu(s^{n-\mu-2} - 2^{n-\mu-2})\right] \,\mathrm{d} s\notag \\
 & =  2^{2-n} \left(  \frac{a_5-\mu 2^{n-\mu-2} }{n-2} \right) + 2^{-\mu}=\frac{2^{2-n}a_5+(n-2-\mu)2^{-\mu}}{n-2}>0.
\end{align}
By \eqref{a4-defn}, \eqref{integral-estimate1} and \eqref{k0-defn}, we get $a_4\in (0,\infty)$ and thus $a_4$ is well-defined. Then by \eqref{eq:varphi},
\begin{equation}\label{laplace-phi-negative}
\Delta \varphi_{\mu}(x)= -a_4\eta_1(x)\le 0\quad\forall x\in\mathbb{R}^n. 
\end{equation}
By \eqref{eq:varphi}  and \eqref{a4-defn}, we get 
\begin{equation}\label{phi-bd}
0<\varphi_{\mu}(r)\le 1\quad\forall r\ge 0.
\end{equation}
Moreover $\varphi_{\mu}(r)$ is a monotone decreasing function of $r\ge 0$ and $\varphi_{\mu}(r)\to 0$ as $r\to\infty$. Note also that by \eqref{eq:varphi} and \eqref{a4-defn}, 
\begin{align}\label{phi-asymptotic behaviour}
\varphi_{\mu}(r)=&1-a_4\left(\int_1^{\infty}s^{1-n}\left(\int_1^s\rho^{n-1} \eta_1(\rho) \,\mathrm{d} \rho \right) \,\mathrm{d} s\right) + a_4\left(\int_r^{\infty}s^{1-n}\left(\int_1^s\rho^{n-1} \eta_1(\rho) \,\mathrm{d} \rho \right) \,\mathrm{d} s\right)\notag\\
=&a_4\left(\int_r^{\infty}s^{1-n}\left(\int_1^s\rho^{n-1} \eta_1(\rho) \,\mathrm{d} \rho \right) \,\mathrm{d} s\right)\notag\\
=&a_4\int_r^{\infty}s^{1-n} \left[a_5 +\mu(s^{n-\mu-2} - 2^{n-\mu-2})\right] \,\mathrm{d} s \notag\\
=&a_4r^{2-n} \left(  \frac{a_5-\mu 2^{n-\mu-2} }{n-2} \right) + a_4r^{-\mu}\quad \forall r >2\notag\\
\to& a_4r^{-\mu}\quad\mbox{ as }r\to\infty.
\end{align}
Hence
\begin{equation}\label{phi'-negative}
\varphi_{\mu}'(r)<0\quad\forall r>0
\end{equation}
and
\begin{align}\label{phi-derivative-asymptotic behaviour}
\varphi_{\mu}'(r)
=&-(n-2)a_4r^{1-n} \left(  \frac{a_5-\mu 2^{n-\mu-2} }{n-2} \right)-\mu a_4r^{-\mu-1}\quad \forall r >2\notag\\
\to&-\mu a_4r^{-\mu-1}\quad\mbox{ as }r\to\infty.
\end{align}
By \eqref{phi-asymptotic behaviour} and \eqref{phi-derivative-asymptotic behaviour} there exists $R_0>2$ such that
\begin{equation}\label{phi-asymptotic behaviour2}
\frac{a_4}{2}r^{-\mu}<\varphi_{\mu}(r)<2a_4r^{-\mu}\quad\forall r\ge R_0
\end{equation}
and
\begin{equation}\label{phi-derivative-asymptotic behaviour2}
-2\mu a_4r^{-\mu-1}<\varphi_{\mu}'(r)<-\frac{\mu a_4}{2}r^{-\mu-1}\quad\forall r\ge R_0.
\end{equation}
Let
\begin{equation*}
L^1(\varphi_{\mu} ; \mathbb{R}^n\setminus\{0\} )=\{v\in L^1_{loc}(\mathbb{R}^n\setminus\{0\}):\int_{\mathbb{R}^n\setminus\{0\}}|v|\varphi_{\mu}\,dx<\infty\}.
\end{equation*}

\begin{theorem}[Weighted $L^1$-contraction]\label{weighted contraction theorem}
  Let $n\ge 3$, $0< m < \frac{n-2}{n}$, $\lambda_1 > \lambda_2 > 0$, $0<\mu<n-2$, $\beta < 0 , \;  \alpha=\frac{2\beta-1}{1-m}$ and let $\gamma$ satisfy \eqref{eq:gamma slpha beta relation6} and \(  \frac{2}{1-m} < \gamma < \frac{n-2}{m}  \). Let $u$, $v$, be solutions of \eqref{eq-fde2-global-except-0} with initial values $u_0, v_0$, respectively which satisfy 
\begin{equation}\label{solns-lower-upper-bd}
 V_{\lambda_1} \le u,v \le V_{\lambda_2}\quad\forall (x,t)\in(\mathbb{R}^n \setminus \{ 0 \} ) \times (0,\infty)
\end{equation} 
where $V_{\lambda_i}$, $i=1,2$ are given by \Cref{eq:V lambda} with $\lambda=\lambda_1,\lambda_2$.
Suppose that $|u_0-v_0| \in L^1( \varphi_{\mu} ; \mathbb{R}^n\setminus\{0\})     $ for some function $\varphi_{\mu}$ given by \Cref{eq:varphi}. Then
      \begin{align}\label{eq:one point singular:52} 
      \int_{\mathbb{R}^n}|u-v|(x,t) \varphi_{\mu}(x) \,\mathrm{d} x \le \int_{\mathbb{R}^n} |u_0-v_0|(x) \varphi_{\mu}(x) \,\mathrm{d} x \quad \forall t>0
      \end{align}
      and
      \begin{align}\label{eq:one point singular:54} 
       \int_{\mathbb{R}^n}(u-v)_+ (x,t) \varphi_{\mu}(x) \,\mathrm{d} x \le \int_{\mathbb{R}^n } (u_0-v_0)_+(x)  \varphi_{\mu}(x) \,\mathrm{d} x \quad \forall t>0.        \end{align}
      
   \end{theorem}
   
\begin{theorem}[Asymptotic large time behaviour of solutions]\label{asymptotic large time behaviour of solution thm}
  Let $n\ge 3$,   $0< m < \frac{n-2}{n}$,  $\beta < 0 , \;  \alpha=\frac{2\beta-1}{1-m}$ and let $\gamma$ satisfy \eqref{eq:gamma slpha beta relation6} and \eqref{gamma-range1}.  Let $u_0$ satisfy \eqref{eq-fde2-initial} and 
  \begin{equation}\label{u0-L1-bd5}
  u_0-A_0|x|^{-\gamma} \in L^1(\varphi_{\mu} ; \mathbb{R}^n )
  \end{equation}
for some constants $A_2\ge A_0\ge A_1>0$ and $0<\mu<n-2$. Let $u$ be the solution of \Cref{eq-fde2-global-except-0}  which satisfies 
\begin{equation}\label{eq:one point singular:62}
V_{\lambda_1} \le u \le V_{\lambda_2} \quad \text{ in } (\mathbb{R}^n  \setminus  \{    0     \}) \times (0,\infty),
\end{equation} 
where $\lambda_i=A_i^{(1-m)\beta}$ for $i=1,2$, and $V_{\lambda_1}$, $V_{\lambda_2}$ are given by  \eqref{eq:V lambda} with $\lambda=\lambda_1, \lambda_2$. Let $\widetilde{u}(y,\tau)$ be given by 
\begin{equation}\label{eq:u tilde}
\widetilde{u}(y,\tau) := t^{\alpha} u (t^{\beta} y , t ) , \quad \tau:= \log t.
\end{equation} 
Then as $\tau \to \infty$, $\widetilde{u}(y,\tau)$ will converge uniformly on each compact subset of $\mathbb{R}^n  \setminus  \{    0     \}$ and in $L^1(\varphi_1; \mathbb{R}^n )$ to $f_{\lambda_0}(y)$ where $\lambda_0=A_0^{(1-m)\beta}$.

\end{theorem}

The plan of the paper is as follows. In \Cref{sec:exist-uniq-sing} we will prove Theorem \ref{existence-self-similar-soln-thm}. In \Cref{sec:asympt-expans-sing} we will prove Theorem \ref{asymptotic expansion of self-similar solution at the origin thm}. Finally we will prove Theorem \ref{weighted contraction theorem} and Theorem \ref{asymptotic large time behaviour of solution thm} in  \Cref{sec:asympt-large-time}. In this paper unless stated otherwise we will assume that \eqref{n-m-relation}, \eqref{eq:gamma range}, \eqref{alpha'-beta'-defn} and \eqref{eq:alpha beta relation3} hold for the rest of the paper.

We start with some definition. For any $0\leq u_0\in L^1_{loc}(\R^n\setminus\{0\}),$  we say that $u$ is a solution of \eqref{eq-fde2-global-except-0}   
if $u>0$ in $(\R^n\setminus\{0\})\times(0, \infty)$ satisfies \eqref{fde} in $ \left(\R^n\setminus\{0\}\right)\times(0, \infty)$ in the classical sense and 
\begin{equation}\label{eq-u-initial-value}
\| u(\cdot, t)-u_0\|_{L^1(K)}\to 0\quad\mbox{ as }t\to 0
\end{equation} 
for any compact set $K\subset \mathbb{R}^n\setminus\{0\}.$  For any $x_0\in\mathbb{R}^n,$ and $R>0,$ we let $B_R(x_0)=\{x\in\mathbb{R}^n: |x-x_0|<R\}$ and $B_R=B_R(0).$ We also let $\mathcal{A}_R= B_{R}\setminus \overline B_{1/R}$ for any $R>1$.

\section{Existence and uniqueness of singular radially symmetric self-similar solution}\label{sec:exist-uniq-sing}

In this section we will use a modification of the technique of \cite{Hui23existe} and fixed point theory to give another proof of Theorem \ref{existence-self-similar-soln-thm}.
Suppose $f$ is a radially symmetric solution of \eqref{eq:one point singular:25} satisfying \eqref{eq:one point singular:27} for some $\eta>0$. Let $w$ be given by  \eqref{w-defn},
\begin{equation}\label{eq:one point singular:1}
\widetilde{w}(s)=w(e^s),\quad s=\log r\quad\forall r\in\mathbb{R}^+
\end{equation}
and  
\begin{equation}\label{z-defn}
z(s)=r \frac{w_r(r)}{w(r)} = \frac{ \round[s] w(e^s)}{w(e^s)},\quad s=\log r\quad\forall r\in\mathbb{R}^+.
\end{equation}
Then by direct computation  (cf. \cite{Hsu12singul} and \cite{Hui23existe}) we have
\begin{equation}\label{eq:one point singular:4}
\left(\frac{w_r}{w}\right)_r+\frac{n-1-\frac{2m\alpha}{\beta}}{r}\cdot\frac{w_r}{w}
+m\left(\frac{w_r}{w}\right)^2
+\frac{\beta r^{-1+\frac{\rho_1}{\beta}}w_r}{w^m}
=\frac{\alpha}{\beta}\cdot\frac{n-2-\frac{m\alpha}{\beta}}{r^2}\quad\forall r>0
\end{equation}
and 
\begin{equation}\label{z-eqn}
z_s+\left(n-2-\frac{2m\alpha}{\beta}\right)z+mz^2+\beta e^{\frac{\rho_1}{\beta}s}\widetilde{w}^{1-m}z=\frac{\alpha}{\beta}\left(n-2-\frac{m\alpha}{\beta}\right)\quad\forall s\in\mathbb{R}.
\end{equation} 
Let 
\begin{equation}\label{eq:one point singular:2}
h(s)=z(s)+\left(\frac{ n-2}{m} - \frac{\alpha}{\beta}\right)\quad\forall s\in\mathbb{R}
\end{equation}
and
\begin{equation}\label{C1-defn}
C_1= \frac{n-2}{m} -\frac{\alpha}{\beta}.
\end{equation}
Then $C_1>0$ and $h$ satisfies
\begin{equation}\label{eq:h_s}
h_s-\left(n-2+\beta' e^{-\frac{\rho_1 s}{ \beta'}}\widetilde{w}^{1-m}\right) h=-\beta' C_1e^{-\frac{\rho_1s}{\beta'}} \widetilde{w}^{1-m}- mh^2  \quad\forall s\in\mathbb{R}.
\end{equation} 
Hence finding a radially symmetric solution $f$ of \eqref{eq:one point singular:25} satisfying \eqref{eq:one point singular:27} is equivalent to solving \eqref{eq:h_s}
with appropriate conditions near $s=\pm\infty$.
Let
\begin{equation*}
 C_2:=\frac{\rho_1}{\beta'} +(1-m) C_1, \quad C_3:= \frac{\beta' C_1\eta_\infty^{1-m}+m}{C_2}
\end{equation*}
and
\begin{equation*}
C_4=\max\left(\frac{2C_3}{C_2},\frac{2^m\beta'C_1}{\eta_\infty^mC_2},\frac{2^m\beta'}{\eta_\infty^mC_2^2}\left(\beta'C_1\eta_{\infty}^{1-m}+C_3^2\right)\right),
\end{equation*}
where $\eta_{\infty}$ is any positive constant. We observe that  $C_2>0$.
\begin{lemma}\label{lemma:local existence of h}
Let $n \ge 3$, $0<m<\frac{n-2}{n}$, $\eta_{\infty}>0$, $b_2\in\mathbb{R}$ and $\alpha$, $\alpha'$, $\beta$, $\beta'$ and $\rho_1$ satisfy \eqref{alpha'-beta'-defn} and \eqref{eq:alpha beta relation3}. 
Then there exists a constant $b_1>0$ such that the equation \Cref{eq:h_s} has a unique solution $h$ in $(b_1,\infty)$ satisfying 
\begin{equation}\label{eq:bound h}
 0<h(s) \le C_3e^{-C_2s}\quad\forall s>b_1
\end{equation} 
with
\begin{equation}\label{eq:one point singular:3}
\widetilde{w}(s) = \eta_\infty \exp \left( -\int_s^{\infty} h(\rho) \,\mathrm{d} \rho \right) e^{-C_1s}\
\end{equation}
for any  $s>b_1$.
\end{lemma}
\begin{proof}
Let $b_1>0$. We define the Banach space
\begin{equation*}
 \mathcal{X}_{b_1} = \left\{ (\widetilde{w},h) \; : \; \widetilde{w},h \in C((b_1,\infty);\mathbb{R})\mbox{ such that}\left \|(\widetilde{w},h)\right \| _{\mathcal{X}_{b_1}}<\infty\right\}
  \end{equation*}
where
\begin{equation*}
\left \|(\widetilde{w},h)\right \| _{\mathcal{X}_{b_1}}= \max \left\{ \left \|\widetilde{w}  \right \| _{L^{\infty} \left(  (b_1,\infty);e^{C_1s} \right)}, \left \|h\right \| _{L^{\infty}\left( (b_1,\infty) ; e^{\frac{C_2}{2}s} \right)}  \right\}
\end{equation*}
and
\begin{equation*}
\left \|v\right \|_{L^{\infty} \left( (b_1,\infty);e^{\lambda s} \right)} = \sup_{b_1 <s<\infty} |v(s) e^{\lambda s} |
\end{equation*}
for any $\lambda \in \mathbb{R}$. Let
\begin{equation}\label{eq:one point singular:6}
\varepsilon_1 := \frac{1}{2}\min \left( 1,\eta_\infty \right)
\end{equation}
and 
\begin{align*}
\mathcal{D}_{b_1}=&\left\{(\widetilde{w},h)\in\mathcal{X}_{b_1} : \|(\widetilde{w},h)-(\eta_\infty e^{-C_1s},0)\|_{\mathcal{X}_{b_1}}\le\varepsilon_1\text{ and }\right.\notag\\
&\qquad\left.\widetilde{w}(s) e^{C_1s} \le \eta_\infty, 0 \le h(s)e^{C_2s} \le C_3\,\forall s>b_1 \right\}.
\end{align*}
Note that there exists a constant $C_5>0$ such that 
\begin{equation}\label{eq:one point singular:15}
1-\exp \left( - \frac{C_3}{C_2} x \right) \le C_5 x \quad \forall x>0.
\end{equation}
We now choose 
\begin{equation}\label{def of b0}
b_1>b_0:=\frac{4}{C_2}\max\left(1,\log (15C_4),
\log\left(\frac{10\eta_\infty+C_3+\beta'\eta_\infty^{1-m}}{C_2}\right),\log\left(\frac{C_3+C_5\eta_\infty}{\3_1}\right)\right). 
\end{equation}
Since $(\eta_\infty e^{-C_1s} , \min(C_3 , \varepsilon_1)e^{-C_2s}) \in \mathcal{D}_{b_1}$, $\mathcal{D}_{b_1}\ne\phi$. For any $(\widetilde{w},h) \in \mathcal{D}_{b_1} $, let
\begin{equation}\label{Phi-map-defn}
\Phi(\widetilde{w},h) := (\Phi_1 (\widetilde{w},h),\Phi_2(\widetilde{w},h)) 
\end{equation}
be given by
\begin{equation}\label{eq:one point singular:8}
\left\{\begin{aligned}
&\Phi_1(\widetilde{w},h)(s)= \eta_\infty\exp\left(-\int_s^{\infty} h(\rho) \,\mathrm{d}\rho \right) e^{-C_1s} 
\\
&\Phi_2(\widetilde{w},h)(s)= \int^{\infty}_s e^{- \beta'\int^{\rho}_s e^{-\frac{\rho_1\rho'}{\beta'}} \widetilde{w}^{1-m}(\rho')\,\mathrm{d} \rho'+(n-2)(s-\rho)}\left( \beta' C_1 e^{-\frac{\rho_1\rho}{\beta'}}\widetilde{w}^{1-m}(\rho) + mh(\rho)^2 \right) \,\mathrm{d} \rho
\end{aligned}\right.
\end{equation} 
for any $s>b_1$. We first prove that the map $\Phi$ is closed, i.e. $\Phi(\mathcal{D}_{b_1}) \subset \mathcal{D}_{b_1} $. Let $(\widetilde{w},h) \in \mathcal{D}_{b_1}$. Then 
\begin{equation}\label{h-range}
0 \le h(s)e^{C_2s} \le C_3  \quad \forall s>b_1.
\end{equation}
By \eqref{eq:one point singular:6} and the definition of $\mathcal{D}_{b_1}$, 
\begin{equation}\label{eq:one point singular:9}
\frac{\eta_\infty}{2}\le \widetilde{w}(s) e^{C_1s} \le \eta_\infty \quad \forall s>b_1.
\end{equation}
By \eqref{eq:one point singular:8} and \eqref{h-range},
\begin{equation}\label{Phi1-upper-bd}
\Phi_1(\widetilde{w},h)(s)e^{C_1s}\le\eta_\infty\quad\forall s>b_1.
\end{equation} 
By \eqref{eq:one point singular:6}, \eqref{def of b0}, \eqref{eq:one point singular:8}, \eqref{h-range} and  \eqref{eq:one point singular:9} we have
\begin{align}\label{eq:one point singular:11} 
0<\Phi_2 (\widetilde{w},h)(s) 
& \le \int^{\infty}_s \left( \beta' C_1 e^{-C_2 \rho} \eta_\infty^{1-m}  + m \varepsilon  _1^2e^{-C_2\rho}\right) \,\mathrm{d} \rho \notag\\
& \le (\beta' C_1\eta_\infty^{1-m}+m)\int_{ s}^{\infty}e^{-C_2\rho}\,d\rho\notag\\
&=C_3e^{-C_2s}\quad\forall s>b_1.
\end{align}
By \eqref{Phi1-upper-bd} and \eqref{eq:one point singular:11}, $\Phi(\widetilde{w},h)\in \mathcal{X}_{b_1}$.
Hence by \eqref{def of b0} and \eqref{eq:one point singular:11}, 
\begin{equation}\label{Phi2-weighted-norm-bd}
0<\Phi_2 (\widetilde{w},h)(s)e^{\frac{C_2}{2}s}\le C_3e^{-\frac{C_2}{2}s}\le\3_1\quad\forall s>b_1\quad
\Rightarrow\quad\|\Phi_2 (\widetilde{w},h)\|_{L^{\infty}\left( (b_1,\infty) ; e^{\frac{C_2}{2}s} \right)} \le\3_1.
\end{equation}
By \eqref{eq:one point singular:15}, \Cref{def of b0} and \eqref{h-range},
\begin{align}\label{Phi1-eta0-difference-bd}
\left|  \Phi_1(\widetilde{w},h)(s) -\eta_\infty e^{-C_1s} \right|e^{C_1s}
&=\eta_\infty\left(1-\exp\left(-\int_s^{\infty} h(\rho) \,\mathrm{d} \rho \right)\right)\notag \\
& \le\eta_\infty \left(1 - \exp \left(-C_3\int_s^{\infty}  e^{-C_2 \rho} \,\mathrm{d} \rho \right)\right) \notag\\
& \le  \eta_\infty \left(1- \exp \left(-\frac{C_3}{C_2} e^{-C_2s} \right)\right)\notag  \\
& \le \eta_\infty C_5e^{-C_2s}\notag \\
& \le \varepsilon_1 \qquad\quad \forall s>b_1.
\end{align} 
Hence
\begin{equation}\label{Phi1-weighted-norm-bd}
\|\Phi_1(\widetilde{w},h)(s)-\eta_\infty e^{-C_1s}\|_{L^{\infty}\left( (b_1,\infty) ; e^{C_1s} \right)}\le \varepsilon_1.
\end{equation}
By \eqref{Phi1-upper-bd}, \eqref{eq:one point singular:11}, \eqref{Phi2-weighted-norm-bd} and \eqref{Phi1-weighted-norm-bd} we get $\Phi(\mathcal{D}_{b_1}) \subset \mathcal{D}_{b_1} $. 

 Next we claim that the map $\Phi\; : \; \mathcal{D}_{b_1} \to \mathcal{D}_{b_1}  $ is a contraction map. In order to prove this claim we let $(\widetilde{w}_1,h_1), \; (\widetilde{w}_2,h_2) \in \mathcal{D}_{b_1}$ and $\delta := \left \|    (\widetilde{w}_1,h_1) - (\widetilde{w}_2,h_2 )      \right \|_{\mathcal{X} _{b_1}}$. Then
 \begin{equation}\label{h-range2}
0 \le h_i(s)e^{C_2s} \le C_3  \quad \forall s>b_1,i=1,2.
\end{equation}
and
\begin{equation}\label{w-tilde-range2}
\frac{\eta_\infty}{2}\le \widetilde{w}_i(s) e^{C_1s} \le \eta_\infty \quad \forall s>b_1,i=1,2.
\end{equation}
By  the mean value theorem for any $s>b_1$ there exists a constant $\xi>0$ between $\int_s^{\infty} h_1 (\rho) \,\mathrm{d} \rho  $  and $\int_s^{\infty} h_2 (\rho) \,\mathrm{d} \rho $ such that
\begin{equation}\label{Phi1-lip-estimate}
|\Phi_1(\widetilde{w}_1,h_1)(s) - \Phi_1(\widetilde{w}_2,h_2)(s)|e^{C_1s}
=\eta_\infty e^{-\xi}\left|\int_s^{\infty}(h_1 (\rho)-h_2 (\rho) )\,\mathrm{d}\rho\right|.
\end{equation}
Hence by \eqref{def of b0} and \eqref{Phi1-lip-estimate},
\begin{align*}
|\Phi_1(\widetilde{w}_1,h_1)(s) - \Phi_1(\widetilde{w}_2,h_2)(s)|e^{C_1s}
\le&\eta_\infty\left\|h_1-h_2\right \| _{L^{\infty}((b_1,\infty) ; e^{\frac{C_2}{2}s})} \int_s^{\infty} e^{-\frac{C_2}{2} \rho} \,\mathrm{d} \rho\notag\\
\le&\frac{2\eta_\infty }{C_2}  e^{-\frac{C_2}{2}s} \delta\quad \forall s>b_1\notag \\
\le&\frac{\delta}{5}\qquad\qquad\quad \forall s>b_1.
\end{align*}   
Thus
\begin{equation}\label{Phi1-contraction} 
\|\Phi_1(\widetilde{w}_1,h_1) -\Phi_1(\widetilde{w}_2,h_2)\|_{L^{\infty}((b_1,\infty);e^{C_1s})}\le\frac{\delta}{5}.
\end{equation}
On the other hand by \eqref{def of b0}, \eqref{eq:one point singular:8}, \eqref{h-range2} and \eqref{w-tilde-range2},
\begin{align}\label{eq:one point singular:7} 
&|\Phi_2(\widetilde{w}_1,h_1)(s)-\Phi_2(\widetilde{w}_2,h_2)(s)|\notag\\ 
\le&\beta' C_1\int^{\infty}_s \exp\left(-\beta'\int^{\rho}_s e^{-\frac{\rho_1}{\beta'}\rho'} \4{w}_1(\rho')^{1-m}\,d\rho'\right)e^{-\frac{\rho_1\rho}{\beta'}}|\widetilde{w}_1(\rho)^{1-m}-\widetilde{w}_2 (\rho)^{1-m}|\,\mathrm{d}\rho\notag\\
&+m\int^{\infty}_s \exp\left(-\beta'\int^{\rho}_s e^{-\frac{\rho_1}{\beta'}\rho'} \4{w}_1(\rho')^{1-m}\,d\rho'\right)|h_1(\rho)^2-h_2(\rho)^2|\,\mathrm{d} \rho\notag\\
 & +\int_s ^{ \infty} \left| e^{-\beta'\int^{\rho}_s e^{-\frac{\rho_1\rho'}{\beta'}}\4{w}_1(\rho')^{1-m}\,d\rho'} - e^{-\int^{\rho}_s \beta' e^{-\frac{\rho_1\rho'}{\beta'}} \widetilde{w}_2^{1-m}(\rho')\,d\rho'} \right|(\beta' C_1 \eta_\infty^{1-m}+ C_3^2 )e^{-C_2 \rho}\,\mathrm{d} \rho\notag \\
\le&\beta' C_1\int^{\infty}_s e^{-\frac{\rho_1\rho}{\beta'}}|\widetilde{w}_1(\rho)^{1-m}-\widetilde{w}_2 (\rho)^{1-m}|\,\mathrm{d}\rho+m\int^{\infty}_s|h_1(\rho)^2 -h_2(\rho) ^2|  \,\mathrm{d} \rho\notag\\
& +\int_s ^{ \infty}\left| e^{-\beta'\int^{\rho}_s  e^{-\frac{\rho_1\rho'}{\beta'}} \widetilde{w}_1(\rho')^{1-m}\,d\rho' } - e^{-\beta'\int^{\rho}_s e^{-\frac{\rho_1\rho'}{\beta'}} \widetilde{w}_2(\rho')^{1-m}\,d\rho'} \right|(\beta' C_1 \eta_\infty^{1-m}+ C_3^2 )e^{-C_2 \rho}\,\mathrm{d} \rho\notag \\
:=&I_1+I_2 +I_3\quad\forall s>b_1.
\end{align}
Note that by \eqref{h-range2} and \eqref{w-tilde-range2} we have
\begin{align}\label{eq:one point singular:18} 
|\widetilde{w}_1(\rho)^{1-m} - \widetilde{w}_2 (\rho)^{1-m} | 
&=\left|\int^1_0 \round[s] (s \widetilde{w}_1(\rho) + (1-s) \widetilde{w}_{2}(\rho)    )^{1-m} \,\mathrm{d} s\right| \notag\\
& \le(1-m)| \widetilde{w}_1(\rho) - \widetilde{w}_2(\rho) | \int^1_0 (s \widetilde{w}_1(\rho) + (1-s) \widetilde{w}_2(\rho))^{-m} \,\mathrm{d} s\notag\\
& \le(1-m)(2/\eta_\infty)^me^{ -(1-m)C_1\rho} \delta\quad\forall\rho>b_1
\end{align}
and
\begin{align}\label{eq:one point singular:19} 
|h_1(\rho)^2  - h_2(\rho)^2  | & \le 2|h_1(\rho) -h_2 (\rho) |\max (h_1(\rho),h_2(\rho)) \le 2 C_3 e^{- \frac{3}{2} C_2 \rho} \delta\quad\forall\rho>b_1.   
\end{align}
By the mean value theorem for any $\rho>s>b_1$ there exists a constant $\xi_1>0$ between 
\begin{equation*}
\beta'\int^{\rho}_s  e^{-\frac{\rho_1\rho'}{\beta'}} \widetilde{w}_1^{1-m}(\rho')\,d\rho'\quad\mbox{ and }\quad 
\beta'\int^{\rho}_s  e^{-\frac{\rho_1\rho'}{\beta'}} \widetilde{w}_2^{1-m}(\rho')\,d\rho'
\end{equation*}
such that
\begin{align}\label{eq:one point singular:21a} 
& \left| e^{-\beta'\int^{\rho}_s  e^{-\frac{\rho_1\rho'}{\beta'}}  \widetilde{w}_1(\rho')^{1-m}\,d\rho'} - 
e^{-\beta'\int^{\rho}_s  e^{-\frac{\rho_1\rho'}{\beta'}} \widetilde{w}_2^{1-m}(\rho')\,d\rho'} \right|\nonumber\\
\le& \beta'e^{-\xi_1}\int^{\rho}_s  e^{-\frac{\rho_1\rho'}{\beta'}}|\widetilde{w}_1(\rho')^{1-m}-\widetilde{w}_2(\rho')^{1-m}|\,d\rho'\quad\forall\rho>s>b_1\notag\\
\le& \beta'\int^{\rho}_s  e^{-\frac{\rho_1\rho'}{\beta'}}|\widetilde{w}_1(\rho')^{1-m}-\widetilde{w}_2(\rho')^{1-m}|\,d\rho'\quad\forall\rho>s>b_1.
\end{align}
By  \eqref{eq:one point singular:18} and \eqref{eq:one point singular:21a},
\begin{align}\label{eq:one point singular:21} 
& \left| e^{-\beta'\int^{\rho}_s  e^{-\frac{\rho_1\rho'}{\beta'}}\widetilde{w}_1(\rho')^{1-m}\,d\rho'} - 
  e^{-\beta'\int^{\rho}_s  e^{-\frac{\rho_1\rho'}{\beta'}} \widetilde{w}_2^{1-m}(\rho')\,d\rho'} \right|\nonumber\\
\le &\beta'  \delta (1-m)(2/\eta_{\infty} )^m \int_s^{\rho} e^{-C_2 \rho'} \,\mathrm{d} \rho' 
 \notag \\
\le&\delta(\beta'/C_2)( 2/\eta_\infty)^me^{-C_2s}\qquad\qquad\quad\forall\rho>s>b_1.
\end{align}  
By \eqref{def of b0}, \eqref{eq:one point singular:7}, \Cref{eq:one point singular:18}, \Cref{eq:one point singular:19} and \Cref{eq:one point singular:21}, we get
\begin{align}\label{I1-upper-bd}
I_1\le\delta\beta' C_1(2/\eta_\infty)^m\int^{\infty}_se^{-C_2 \rho}\,\mathrm{d}\rho
  \le C_4 e^{-C_2s}\delta\quad\forall s>b_1
\end{align}
\begin{equation}\label{I2-upper-bd}
I_2\le 2mC_3\delta \int^{\infty}_se^{-\frac{3}{2}C_2\rho}\,\mathrm{d} \rho
\le 2\delta(C_3/C_2) e^{-\frac{3}{2}C_2s}\le C_4 e^{-\frac{3}{2}C_2s}\delta\quad\forall s>b_1
\end{equation}
and
\begin{equation}\label{I3-upper-bd}
  I_3\le\delta (\beta'/C_2)(2/\eta_\infty)^me^{-C_2s}\int^{\infty}_s(\beta'C_1\eta_\infty^{1-m}+C_3^2)e^{-C_2\rho}\,\mathrm{d} \rho\le C_4e^{-2C_2s}\delta \quad\forall s>b_1. 
\end{equation}
By \eqref{def of b0}, \eqref{eq:one point singular:7}, \eqref{I1-upper-bd}, \eqref{I2-upper-bd} and \eqref{I3-upper-bd},
\begin{align}\label{Phi2-contraction} 
&|\Phi_2(\widetilde{w}_1,h_1)(s) - \Phi_2(\widetilde{w}_2,h_2)(s)|e^{\frac{C_2}{2}s}\le 3C_4e^{-\frac{C_2}{2}s}\delta\le\frac{\delta}{5} \quad\forall s>b_1\notag\\
\Rightarrow\quad&\|\Phi_2(\widetilde{w}_1,h_1)-\Phi_2(\widetilde{w}_2,h_2)\|_{L^{\infty}\left((b_1,\infty);e^{ \frac{C_2}{2}s}\right)}\le\frac{\delta}{5}. 
\end{align}
By \eqref{Phi1-contraction}  and \eqref{Phi2-contraction},
\begin{equation*}
\|\Phi(\widetilde{w}_1,h_1)-\Phi(\widetilde{w}_2,h_2)\|_{\mathbb{X}_{b_1}}\le\frac{\delta}{5}.
\end{equation*}
Hence the map $\Phi:\mathcal{D}_{b_1} \to \mathcal{D}_{b_1}$ is a contraction map with Lipschitz constant less than $1/5$. Since $\mathcal{D}_{b_1}$ is a complete metric space, by the contraction mapping theorem there exists a unique fixed point  $(\widetilde{w},h)\in\mathcal{D}_{b_1}$. Then $(\widetilde{w},h)=\Phi(\widetilde{w},h)$. Hence
\begin{equation}\label{h-integral-representation}
h(s)=\int^{\infty}_s e^{- \beta'\int^{\rho}_s e^{-\frac{\rho_1\rho'}{\beta'}} \widetilde{w}^{1-m}(\rho')\,\mathrm{d} \rho'+(n-2)(s-\rho)}\left( \beta' C_1 e^{-\frac{\rho_1\rho}{\beta'}}\widetilde{w}^{1-m}(\rho) + mh(\rho)^2 \right) \,\mathrm{d} \rho\quad\forall s>b_1
\end{equation} 
with 
\begin{equation}\label{tilde-w-eqn}
\widetilde{w}(s)= \eta_\infty\exp\left(-\int_s^{\infty} h(\rho) \,\mathrm{d}\rho \right) e^{-C_1s}\quad\forall s>b_1. 
\end{equation}
Since $(\widetilde{w},h)\in\mathcal{D}_{b_1}$, \eqref{h-range} holds and
\begin{equation}\label{tilde-w-ineqn}
\widetilde{w}(s) e^{C_1s} \le \eta_\infty \quad \forall s>b_1.
\end{equation}
By \eqref{h-integral-representation}, $h(s)>0$ for any $s>b_1$. This together with \eqref{h-range} implies that \eqref{eq:bound h} holds.
Differenetiating \eqref{h-integral-representation} with respect to $s$, $s>b_1$,  we get that $h$ satisfies 
\eqref{eq:h_s} in $(b_1,\infty)$ with $\widetilde{w}$ given by \eqref{tilde-w-eqn}. 

We will now prove the uniqueness of solution of \eqref{eq:h_s}.
Suppose $h_1$ is any solution of 
\begin{equation}\label{eq:h1_s}
h_{1,s}-\left(n-2+\beta' e^{-\frac{\rho_1 s}{ \beta'}}\widetilde{w}_1^{1-m}\right) h_1=-\beta' C_1e^{-\frac{\rho_1s}{\beta'}} \widetilde{w}_1^{1-m}- mh_1^2 
\end{equation}  
in $(b_1,\infty)$ satisfying
\begin{equation}\label{eq:bound h1}
 0<h_1(s) \le C_3e^{-C_2s}\quad\forall s>b_1
\end{equation}
with 
\begin{equation}\label{tilde-w1-eqn}
\widetilde{w}_1(s)= \eta_\infty\exp\left(-\int_s^{\infty} h_1(\rho) \,\mathrm{d}\rho \right) e^{-C_1s}\quad\forall s>b_1. 
\end{equation}
Then by  \eqref{eq:bound h1} and \eqref{tilde-w1-eqn},
\begin{equation}\label{tilde-w1-ineqn}
\widetilde{w}_1(s) e^{C_1s}\le\eta_\infty \quad \forall s>b_1.
\end{equation}
By \eqref{eq:bound h1}, \eqref{tilde-w1-eqn} and an argument similar to the proof of \eqref{Phi1-eta0-difference-bd} we have
\begin{align}\label{w1-tilde-eta0-difference-bd}
&|\widetilde{w}_1(s) -\eta_\infty e^{-C_1s}|e^{C_1s}\le \varepsilon_1\quad \forall s>b_1\notag\\
\Rightarrow\quad&\|\widetilde{w}_1(s)-\eta_\infty e^{-C_1s}\|_{L^{\infty}\left( (b_1,\infty) ; e^{C_1s} \right)}\le \varepsilon_1.
\end{align}
By \eqref{def of b0} and \eqref{eq:bound h1},
\begin{equation}\label{h1-bd1}
0<h_1(s)e^{\frac{C_2}{2}s} \le C_3e^{-\frac{C_2}{2}s}\le\3_1\quad\forall s>b_1.
\end{equation}
By \eqref{tilde-w1-ineqn}, \eqref{w1-tilde-eta0-difference-bd} and \eqref{h1-bd1}, $(\widetilde{w}_1,h_1)\in\mathcal{D}_{b_1}$. By \eqref{eq:h1_s},
\begin{align}\label{eq:h1_s-2}
&\left(\exp\left(-\beta'\int^{\rho}_s e^{-\frac{\rho_1\rho'}{\beta'}} \widetilde{w}_1^{1-m}(\rho')\,\mathrm{d} \rho'+(n-2)(s-\rho)\right)h_1(\rho)\right)_{\rho}\notag\\
=&\exp\left(-\beta'\int^{\rho}_s e^{-\frac{\rho_1\rho'}{\beta'}} \widetilde{w}_1(\rho')^{1-m}\,\mathrm{d} \rho'+(n-2)(s-\rho)\right)\left(\beta' C_1e^{-\frac{\rho_1\rho}{\beta'}} \widetilde{w}_1(\rho)^{1-m}+mh_1(\rho)^2\right) 
\end{align} 
holds for any $\rho>s>b_1$.
Note that
\begin{align}\label{integral-factor-to-zero}
\exp\left(-\beta'\int^{\rho}_s e^{-\frac{\rho_1\rho'}{\beta'}} \widetilde{w}_1^{1-m}(\rho')\,\mathrm{d} \rho'+(n-2)(s-\rho)\right)
\le&\exp\left((n-2)(s-\rho)\right)\quad\forall \rho>s>b_1\notag\\
\to&0\quad\mbox{ as }\rho\to\infty.
\end{align}
Integrating \eqref{eq:h1_s-2} over $\rho\in (s,\infty)$, by \eqref{eq:bound h1} and \eqref{integral-factor-to-zero} we get that $h_1$ also satisfies \eqref{h-integral-representation} in $(b_1,\infty)$ with $\widetilde{w}$ being replaced by $\widetilde{w}_1$ given by \eqref{tilde-w-eqn}. Hence by the uniqueness of solution of \eqref{h-integral-representation} in $\mathcal{D}_{b_1}$, we get $(\widetilde{w}_1,h_1)=(\widetilde{w},h)$ and the lemma follows.

\end{proof}

\begin{lemma}\label{h-bded-lem}
Let $n \ge 3$, $0<m<\frac{n-2}{n}$, $\eta_{\infty}>0$, $b_2\in\mathbb{R}$ and $\alpha$, $\alpha'$, $\beta$, $\beta'$ and $\rho_1$ satisfy \eqref{alpha'-beta'-defn} and \eqref{eq:alpha beta relation3}. Suppose $h$ is a solution of \Cref{eq:h_s} in $(b_2,\infty)$ with $\widetilde{w}$ given by \eqref{eq:one point singular:3} which satisfies \eqref{eq:bound h} in $(b_1,\infty)$ for some constant $b_1>b_2$. Then
\begin{equation}\label{h-bd-11}
0<h(s)<\frac{n-2}{m}-\frac{\alpha}{\beta}\quad\forall s>b_2.
\end{equation}
\end{lemma}
\begin{proof}
We first observe that by \Cref{eq:h_s} $h$ satisfies \eqref{h-integral-representation} for any $s>b_2$. Hence $h(s)>0$ for any $s>b_2$. We will give two different proofs for the right hand side inequality of \eqref{h-bd-11}.

\noindent{\bf Method 1:}

\noindent By \eqref{eq:bound h} there exists a constant $b_3>b_1$ such that
\begin{equation*}
h(s)<\frac{n-2}{m}-\frac{\alpha}{\beta}=C_1\quad\forall s>b_3.
\end{equation*}
Let $b_4=\inf\{b>b_2:h(s)<C_1\quad\forall s>b\}$. Then $b_4$ exists and $b_4<b_3$. Suppose $b_4>b_2$. Then
\begin{equation}\label{h-upper-bd12}
h(s)<C_1\quad\forall s>b_3\quad\mbox{ and }\quad h(b_4)=C_1.
\end{equation}
By \Cref{eq:h_s} and \eqref{h-upper-bd12},
\begin{align*}
h_s&-\left((n-2)+\beta' e^{-\frac{\rho_1 s}{ \beta'}}\widetilde{w}^{1-m}\right) h\ge -\beta' C_1e^{-\frac{\rho_1s}{\beta'}} \widetilde{w}^{1-m}- mC_1h  \quad\forall s>b_4\notag\\
\Rightarrow\qquad\quad h_s&-\beta' e^{-\frac{\rho_1 s}{ \beta'}}\widetilde{w}^{1-m} h\ge -\beta' C_1e^{-\frac{\rho_1s}{\beta'}} \widetilde{w}^{1-m}\quad\forall s>b_4\notag\\
\Rightarrow\quad h(s)\le&\int^{\infty}_s e^{- \beta'\int^{\rho}_s e^{-\frac{\rho_1\rho'}{\beta'}} \widetilde{w}^{1-m}(\rho')\,\mathrm{d} \rho'}\left( \beta' C_1 e^{-\frac{\rho_1\rho}{\beta'}}\widetilde{w}^{1-m}(\rho)  \right) \,\mathrm{d} \rho\quad\forall s\ge b_4\notag\\
=&-C_1\int^{\infty}_s \frac{\partial}{\partial \rho}\left( e^{- \beta'\int^{\rho}_s e^{-\frac{\rho_1\rho'}{\beta'}} \widetilde{w}^{1-m}(\rho')\,\mathrm{d} \rho'}\right) \,\mathrm{d} \rho\quad\forall s\ge b_4\notag\\
=&C_1\left(1-e^{- \beta'\int_s^{\infty} e^{-\frac{\rho_1\rho'}{\beta'}} \widetilde{w}^{1-m}(\rho')\,\mathrm{d} \rho'}\right)\quad\forall s\ge b_4\notag\\
\Rightarrow\quad h(s)<&C_1\quad\forall s\ge b_4
\end{align*}
which contradicts \eqref{h-upper-bd12}. Hence $b_4=b_2$ and the right hand side inequality of \eqref{h-bd-11} follows.

\noindent{\bf Method 2:}

\noindent Let
\begin{equation}\label{w-defn-w-tilde}
w(r)=\4{w}(s)\quad \forall r=e^s, s>b_2, 
\end{equation}
and $f$ be given by
\begin{equation}\label{f-define-w}
f(r)=r^{-\alpha/\beta}w(r)\quad\forall r>e^{b_2}.
\end{equation}
We also let $z$ be given by \eqref{eq:one point singular:2} and 
\begin{equation}\label{g-defn}
g(r)=r^{-\frac{n-2}{m}}f(r^{-1})\quad\forall r>e^{b_2}.
\end{equation}
Differentiating \eqref{eq:one point singular:3} we respect to $s$ we get
\begin{align}\label{h-wr-eqn}
&\4{w}_s(s)=(h(s)-C_1)\4{w}(s)\quad\forall s>b_2\notag\\
\Rightarrow\quad&h(s)=\frac{\4{w}_s(s)}{\4{w}(s)}+C_1=\frac{rw_r(r)}{w(r)}+C_1\quad\forall r=e^s, s>b_2.
\end{align}
Putting \eqref{h-wr-eqn} in \eqref{eq:h_s} we get that $w$ satisfies \eqref{eq:one point singular:4} in $(b_2,\infty)$. Putting \eqref{f-define-w} in \eqref{eq:one point singular:4} we get that $f$ satisfies \begin{equation}\label{eq:f-ODE}
(f^m/m)_{rr}+\frac{n-1}{r}(f^m/m)_r + \alpha f + \beta rf_r =0,\quad f>0,
\end{equation} 
in $(b_2,\infty)$. By \eqref{g-defn}, \eqref{eq:f-ODE} and a direct computation (cf. \cite{HK17asympt}), $g$ satisfies
\begin{equation}\label{eq-fde-inversion}
(g^m)''+\frac{n-1}{r}(g^m)' +r^{\frac{n-2-nm}{m}-2}(\tilde\alpha g+\tilde\beta r g_{r})=0,\quad g>0,
\end{equation}
in $(b_2,\infty)$ where
\begin{equation*}\label{eq-tilde-alpha-beta}
\tilde\beta=-\beta,\quad\mbox{and}\quad\tilde\alpha=\alpha-\frac{n-2}{m}\,\beta.
\end{equation*}
Thus
\begin{equation}\label{eq-tilde-alpha-beta-0}
\tilde\alpha>0,\quad\tilde\beta>0,\quad \frac{\tilde\alpha}{\tilde\beta}=\frac{n-2}{m}-\frac{\alpha}{\beta}\
=C_1 \in \left(0,\frac{n-2}{m}\right).
\end{equation}
 By \eqref{f-define-w} and \eqref{g-defn},
\begin{equation}\label{w-g-relation}
w(r)=r^{\frac{\alpha}{\beta}-\frac{n-2}{m}}g(r^{-1})\quad\forall r>e^{b_2}.
\end{equation}
Then by \eqref{w-g-relation},
\begin{equation}\label{w-g-derivative-relation}
\frac{rw_r(r)}{w(r)}=\frac{\alpha}{\beta}-\frac{n-2}{m}-\frac{\rho g_{\rho}(\rho)}{g(\rho)}\quad\forall \rho=r^{-1}, r>e^{b_2}.
\end{equation}
By \eqref{h-wr-eqn} and \eqref{w-g-derivative-relation},
\begin{equation}\label{h-g-relation}
h(s)=-\frac{\rho g_{\rho}(\rho)}{g(\rho)}\quad\forall \rho=e^{-s}, s>b_2.
\end{equation}
By \eqref{eq:one point singular:3}, \eqref{w-defn-w-tilde} and \eqref{w-g-relation},
\begin{align}
&g(\rho)= \eta_\infty\exp\left(-\int_{\log (1/\rho)}^{\infty} h(\rho') \,\mathrm{d}\rho' \right) \quad\forall \rho=e^{-s},s>b_2\label{g-integral-representation}\\
\Rightarrow\quad&\lim_{\rho\to 0}g(\rho)= \eta_\infty.\label{g-at-origin}
\end{align}
By \eqref{g-at-origin} we can extend $g$ to a continuous function on $[0,e^{-b_2})$ by defining
\begin{equation}\label{g-at-origin2}
g(0)=\eta_{\infty}.
\end{equation}
By \eqref{eq:bound h} and \eqref{g-integral-representation},
\begin{align}\label{g-derivative-limit}
&g_{\rho}(\rho)= -\eta_\infty\rho^{-1} h(\log(1/\rho))\exp\left(-\int_{\log (1/\rho)}^{\infty} h(\rho') \,\mathrm{d}\rho' \right) \quad\forall \rho=e^{-s},s>b_2\notag\\
\Rightarrow\quad&\lim_{\rho\to 0}\rho g_{\rho}(\rho)= -\eta_\infty \lim_{\rho\to 0}h(\log(1/\rho))\cdot\lim_{\rho\to 0}\exp\left(-\int_{\log (1/\rho)}^{\infty} h(\rho') \,\mathrm{d}\rho' \right)= -\eta_\infty\cdot 0=0.
\end{align}
By \eqref{eq-fde-inversion}, \eqref{eq-tilde-alpha-beta-0}, \eqref{g-at-origin2} and \eqref{g-derivative-limit} the hypothesis in Lemma 2.1 of \cite{HK17asympt} are satisfied. Hence by Lemma 2.1 of \cite{HK17asympt},
\begin{equation}\label{g-g-rho-ineqn}
C_1g(\rho)+\rho g_{\rho}(\rho)>0\quad\forall \rho=e^{-s},s>b_2.
\end{equation}
By \eqref{h-g-relation} and \eqref{g-g-rho-ineqn} we get the right hand side of \eqref{h-bd-11} holds and the lemma follows.

\end{proof}

\begin{theorem}\label{theorem:global existence of h}
Let $n \ge 3$, $0<m<\frac{n-2}{n}$, $\eta_{\infty}>0$, $b_2\in\mathbb{R}$ and $\alpha$, $\alpha'$, $\beta$, $\beta'$ and $\rho_1$ satisfy \eqref{alpha'-beta'-defn} and \eqref{eq:alpha beta relation3}.  Then the equation \Cref{eq:h_s} has a unique positive solution $h$ in $\mathbb{R}$ which satisfies \eqref{h-bd-11}  in $\mathbb{R}$ with $\widetilde{w}$ satisfying \eqref{eq:one point singular:3} in $\mathbb{R}$. Moreover $h$ satisfies \eqref{eq:bound h} in $(b_1,\infty)$ for some constant $b_1\in\mathbb{R}$. 
\end{theorem}
\begin{proof}
Since the proof of the theorem is similar to the proof of Theorem 2.2 of \cite{Hui23existe}, we will only sketch the argument here. By Lemma \ref{lemma:local existence of h} there exists a constant $b_1>0$ such that the equation \Cref{eq:h_s} has a unique solution $h$ in $(b_1,\infty)$ with $\widetilde{w}$ satisfying \eqref{eq:one point singular:3} in $(b_1,\infty)$ which satisfies \eqref{eq:bound h} in $(b_1,\infty)$. Let $(b_2,\infty)$ be the maximal interval of existence of solution of \Cref{eq:h_s} with $\widetilde{w}$ satisfying \eqref{eq:one point singular:3} in $(b_2,\infty)$. By Lemma \ref{h-bded-lem} \eqref{h-bd-11} holds. Suppose $b_2>-\infty$. Then   there exists a 
sequence $\{r_i\}_{i=1}^{\infty}\subset (b_2,\infty)$, $r_i\to b_2$ as $i\to\infty$,  such that
\begin{equation*}
h(r_i)\to 0\quad\mbox{ as }i\to\infty
\end{equation*} 
or
\begin{equation*}
h(r_i)\to\infty\quad\mbox{ as }i\to\infty
\end{equation*}
which contradicts \eqref{h-bd-11}. Hence $b_2=-\infty$ and the theorem follows.
\end{proof}

By Lemma \ref{h-bded-lem}, Theorem \ref{theorem:global existence of h} and the proof of Lemma \ref{h-bded-lem} we have the following result.

\begin{corollary}\label{cor:f-existence-infty}
Let $n \ge 3$, $0<m<\frac{n-2}{n}$, $\eta_{\infty}>0$ and $\alpha$, $\beta$ and $\rho_1$ satisfy \eqref{eq:alpha beta relation3}. Then there exists  a unique solution $f$ of \eqref{eq:f-ODE} in $(0,\infty)$
satisfying
\begin{equation}\label{f-infinty-limit-behaviour}
\lim_{r \to \infty }r^{\frac{n-2}{m}}f(r)=\eta_\infty
\end{equation}
and
\begin{equation}\label{f-origin-limit-behaviour}
\lim_{r \to 0}r^{\alpha/\beta}f(r)=D^{-1}(\eta_\infty)
\end{equation}
for some constant $D^{-1}(\eta_\infty)>0$ with
\begin{equation}\label{eq:f-fr-bd}
-\frac{n-2}{m}<\frac{rf_r(r)}{f(r)} <-\frac{\alpha}{\beta}\quad\forall r>0.
\end{equation} 
\end{corollary}
\begin{proof}
Let $h$ and $\widetilde{w}$ be given by Theorem \ref{theorem:global existence of h}. Let $w$ and $f$ be given by \eqref{w-defn-w-tilde} and \eqref{f-define-w}. Then by the proof of Lemma \ref{h-bded-lem}  $f$ satisfies  \eqref{eq:f-ODE} in $(0,\infty)$. By \eqref{f-define-w}, we have
\begin{equation}\label{w-wr-f-fr-relation}
\frac{rw_r(r)}{w(r)}=\frac{rf_r(r)}{f(r)}+\frac{\alpha}{\beta}\quad\forall r>0.
\end{equation} 
By \eqref{h-bd-11}, \eqref{h-wr-eqn} and \eqref{w-wr-f-fr-relation} we get \eqref{eq:f-fr-bd}.
By \eqref{eq:one point singular:3}, \eqref{w-defn-w-tilde} and \eqref{f-define-w},
\begin{equation}\label{eq:f-infty-behaviour}
f(r)r^{\frac{n-2}{m}}= \eta_\infty \exp \left( -\int_{\log r}^{\infty} h(\rho) \,\mathrm{d} \rho \right)\quad\forall r>0
\end{equation}
Letting $r\to\infty$ in \eqref{eq:f-infty-behaviour}. Then by Corollary 2.6 and the proof of Lemma 3.1
of \cite{HK17asympt} we get that \eqref{f-origin-limit-behaviour} holds for some constant $D^{-1}(\eta_\infty)>0$.
The uniqueness of such solution also follows from the results of \cite{HK17asympt}
and the corollary follows.
\end{proof}

\begin{remark}
  From the uniqueness, we can now practically identify initial condition constants \(  \eta = D^{-1}_{\eta_{\infty}}  \) and \(  \eta_{\infty} = D(\eta)  \).
\end{remark}
By Corollary \ref{cor:f-existence-infty}, Proposition 2.5 of \cite{HK17asympt} and the same argument as the proof of Lemma 3.1 of \cite{HK17asympt} we get Theorem  \ref{existence-self-similar-soln-thm}.

\section{Asymptotic expansion of singular self-similar solution at the origin}\label{sec:asympt-expans-sing}
In this section we will use a modification of the technique of K.M.~Hui and Sunghoon Kim \cite{HK19singul} to prove the second order asymptotic expansion of  the singular radially symmetric self-simliar solution near the origin. In this section we will let $f$ be the unique radially symmetric solution of \eqref{eq:one point singular:25} satisfying \eqref{eq:one point singular:27} for some $\eta>0$. We let $\rho = r^{\frac{\rho_1}{\beta'}}$ and $\overline{w}(\rho)= w(r)$ for any $r>0$. Then by \eqref{eq:one point singular:4} and a direct computation $\overline{w}(\rho)$ satisfies (cf. (2.1) of \cite{HK19singul})
\begin{equation}
\label{eq:one point singular:26}
\left(  \frac{\overline{w}_{\rho}}{   \overline{w} } \right)_{\rho} + m \left(  \frac{ \overline{w}_{\rho} }{ \overline{w}}  \right)^2 + \frac{a_1 }{\rho} \cdot \frac{  \overline{w} _{\rho} }{ \overline{w}} + \frac{a_2}{ \rho^2 } \cdot \frac{\overline{w}_{\rho}}{ \overline{w}^m} = \frac{a_3}{\rho^2 } \quad \forall \rho>0
\end{equation}
where the constants $a_1, a_2, a_3$, are given by \eqref{eq:one point singular:30}. Note that $a_2<0$ and by \eqref{eq:alpha beta relation3} $a_3>0$. By \eqref{eq:one point singular:27},
\begin{equation}\label{w-bar-at-origin}
\lim_{\rho\to 0^+}\overline{w}(\rho)=\eta.
\end{equation} 
Note that by Theorem \ref{existence-self-similar-soln-thm}, \eqref{f-infty-behaviour} holds for some constant $D(\eta)>0$. Hence by \eqref{eq:alpha beta relation3} and \eqref{f-infty-behaviour},
\begin{equation}\label{w-bar-to-zero-at-infty}
\lim_{\rho\to\infty}\4{w}(\rho)=\lim_{r\to\infty}w(r)=\lim_{r\to\infty}r^{\frac{n-2}{m}}f(r)\cdot r^{\frac{\alpha}{\beta} - \frac{n-2}{m}}= D(\eta)\cdot\lim_{r\to\infty}r^{\frac{\alpha}{\beta}-\frac{n-2}{m}}=0.
\end{equation}

\begin{lemma} \label{lemma:one point singular:1}
Let $n \ge 3$, $0<m < (n-2)/n$, $\eta>0$ and $\alpha$, $\alpha'$, $\beta$, $\beta'$, $\rho_1$, satisfy  \eqref{alpha'-beta'-defn} and \eqref{eq:alpha beta relation3}. Then 
\begin{equation}\label{w-derivative-negative}
\overline{w}_{\rho}(\rho) <0\quad\forall \rho>0.
\end{equation}
\end{lemma}
\begin{proof}
Suppose \eqref{w-derivative-negative} does not hold. Then there exist a constant $\rho_2 >0$ such that 
\begin{equation}\label{eq:one point singular:33}
\overline{w}_{\rho}(\rho_2) \ge 0.
\end{equation}
Then by  \eqref{eq:one point singular:30}, \Cref{eq:one point singular:26} and \Cref{eq:one point singular:33} we have
\begin{equation}
\label{eq:one point singular:34}
\left(  \rho^{a_1} \cdot \overline{w}^m \cdot \frac{\overline{w}_{\rho}}{ \overline{w}}  \right)_{\rho} (\rho_2) = -a_2 \rho_2 ^{a_1-2} \overline{w}_{\rho} (\rho_2) + a_3 \rho_2^{a_1 -2} \overline{w}(\rho_2) ^m > 0. 
\end{equation}
Thus there exists a  constant $\delta_1>0$ such that  
\begin{align*}
&\rho^{a_1} \cdot \overline{w}^m \cdot \frac{\overline{w}_{\rho}}{ \overline{w}} (\rho) >0\quad\forall\rho \in (\rho_2,\rho_2 + \delta_1)\notag\\
\Rightarrow\quad&\overline{w}_{\rho}(\rho) >0 \qquad\qquad\quad \forall \rho \in (\rho_2,\rho_2 +\delta_1).
\end{align*}
Let $(\rho_2,\rho_3)$ be the maximal interval such that
\begin{equation}\label{eq:one point singular:37}
\overline{w}_{\rho}(\rho) >0 \quad \forall \rho \in (\rho_2,\rho_3).
\end{equation} 
By \Cref{eq:one point singular:26} and \Cref{eq:one point singular:37},
\begin{equation}\label{eq:one point singular:39}
(\rho^{a_1} (\overline{w}^m/m)_{\rho})_{\rho}(\rho)=\left(\rho^{a_1}\cdot\overline{w}^m\cdot\frac{\overline{w}_{\rho} }{ \overline{w}}  \right)_{\rho} (\rho) \ge a_3 \rho^{a_1-2} \overline{w}(\rho)^m  >0 \quad \forall \rho \in (\rho_2,\rho_3).
\end{equation}
Integrating \eqref{eq:one point singular:39} over $(\rho_2,\rho)$, $\rho_2<\rho<\rho_3$,  
\begin{equation}\label{w-bar-ineqn1} 
\rho_2^{a_1} (\overline{w}^m/m)_{\rho}(\rho_2)
\le\left\{\begin{aligned}
&\rho^{a_1}(\overline{w}^m/m)_{\rho} (\rho)+\frac{a_3 \overline{w}(\rho_2)^m }{ 1-a_1 } ( \rho^{a_1-1} -\rho_2 ^{ a_1-1} ) \quad \forall \rho_2 < \rho < \rho_3\quad\mbox{ if }a_1\ne 1\\
&\rho^{a_1}(\overline{w}^m/m)_{\rho} (\rho)+a_3 \overline{w}(\rho_2)^m\log\left(\frac{\rho_2}{\rho}\right)\qquad\quad \forall \rho_2 < \rho < \rho_3\quad\mbox{ if }a_1=1.
\end{aligned}\right.
\end{equation}
Dividing  \eqref{w-bar-ineqn1} by $\rho^{a_1}$ and integrating over $(\rho_2,\rho)$, $\rho_2<\rho<\rho_3$, 
 \begin{equation}\label{eq:one point singular:40} 
\overline{w}(\rho)^m\ge\left\{\begin{aligned}
&\overline{w}(\rho_2)^m + E_1 (\rho^{1-a_1} -\rho_2^{1-a_1})  +  \frac{a_3 m\overline{w}(\rho_2)^m }{ 1-a_1 } \log \left( \frac{\rho_2}{\rho} \right) \qquad\quad \forall \rho_2 < \rho < \rho_3\quad\mbox{ if }a_1\ne 1 \\
&\overline{w}(\rho_2)^m + E_2\log \left( \frac{\rho}{\rho_2} \right)+a_3m\4{w}(\rho_2)^m\left[(\log\rho)^2-(\log\rho_2)^2\right]\quad \forall \rho_2 < \rho < \rho_3\quad\mbox{ if }a_1=1
\end{aligned}\right.
 \end{equation}
 where
\begin{equation}\label{E1-defn} 
E_1=\frac{m}{1-a_1}\left(\rho_2^{a_1}(\overline{w}^m/m)_{\rho} (\rho_2)+\frac{a_3 \rho_2 ^{a_1-1}\overline{w}(\rho_2) ^m}{1-a_1}\right)\ge \frac{a_3m \rho_2 ^{a_1-1}\overline{w} (\rho_2) ^m}{(1-a_1)^2}>0\quad\mbox{ if }a_1<1 
\end{equation}
 and
 \begin{equation*}
 E_2=\rho_2m(\overline{w}^m/m)_{\rho}(\rho_2)-a_3m\overline{w}(\rho_2)^m\log\rho_2.
 \end{equation*}
 If $\rho_3 = \infty$, then by letting $\rho \to \infty$ in \Cref{eq:one point singular:40}, by \eqref{E1-defn} we get
 \begin{equation*}
  \overline{w}(\rho) \to \infty \quad \text{as} \quad \rho \to \infty.
 \end{equation*}
which contradicts \eqref{w-bar-to-zero-at-infty}. Hence $\rho_3<\infty$  and
\begin{equation}\label{eq:one point singular:38}
\overline{w}_{\rho}(\rho_3)=0.
\end{equation}
By \Cref{eq:one point singular:39} and \Cref{eq:one point singular:38},
\begin{equation}\label{eq:one point singular:41}
\overline{w}_{\rho}(\rho) <0  \quad  \forall \rho \in (\rho_2,\rho_3)
\end{equation}
which contradicts  \Cref{eq:one point singular:37}. Hence no such $\rho_2 >0$ exists and the lemma follows.

\end{proof}

\begin{lemma}\label{lm:lemma 3.2}\label{w-rho-derivative-limit-at-origin-lem}
Let $n \ge 3$, $0<m < (n-2)/n$, $\eta>0$ and $\alpha$, $\alpha'$, $\beta$, $\beta'$, $\rho_1$, satisfy  \eqref{alpha'-beta'-defn} and \eqref{eq:alpha beta relation3}.  Then 
\begin{equation}\label{derivative of barw at 0}
\lim_{\rho \to 0^+} \overline{w}_{\rho} (\rho) = \frac{a_3}{a_2}\eta^m = \frac{m \alpha^2-\alpha\beta(n-2)}{ \beta^2 \rho_1}  \eta^m
\end{equation}
and 
\begin{align}\label{f-derivative-at-origin}
\lim_{r \to 0^+} r^{\frac{\alpha}{\beta} + 1 } f'(r) = -\frac{\alpha}{\beta} \eta 
\end{align}
where $a_1,\ a_2 $ and $a_3$  are constants given by \Cref{eq:one point singular:30}. Hence $\overline{w}$ can be extended to a function in  \(  C^1([0,\infty)  \) by letting  \(  \overline{w}(0) = \eta  \) and  \(\overline{w}_{\rho}(0) = \frac{a_3}{a_2} \eta^m  \).
\end{lemma}

\begin{proof} We will use a modification of the proof of Lemma 2.2 of \cite{HK19singul} to prove the lemma. Note that the main difference between the case in \cite{HK19singul} and our case is that  $a_2$ is positive in \cite{HK19singul} while $a_2$ is negative in this paper. Let \begin{equation*}
q(\rho)= \frac{1}{\overline{w}_{\rho}(\rho)}\quad\forall\rho>0. 
\end{equation*}
Then by (2.15) of \cite{HK19singul} $q$ satisfies
\begin{align} 
q_{\rho} & = - \frac{1-m}{\overline{w}(\rho)} + \frac{q(\rho)}{ \rho} \left[  a_1 + \frac{ \overline{w}(\rho)}{ \rho} \left(\frac{a_2}{ \overline{w}(\rho)^m}- a_3 q(\rho) \right) \right]\quad\forall\rho>0. \label{eq:one point singular:31} 
\end{align}
By \eqref{w-bar-at-origin} and Lemma \ref{lemma:one point singular:1} there exists a constant $\rho_2>0$ such that 
 \begin{equation}\label{eq:one point singular:42}
 \frac{\eta}{2}< \overline{w}(\rho) < \eta  \quad \forall 0<\rho < \rho_2.
 \end{equation}
 We claim that there exists a constant $\rho_0>0$ such that
 \begin{equation}\label{eq:one point singular:32}
  \frac{2^{1+m} a_2}{ a_3} \eta^{-m} \le q(\rho) \le \frac{2^{m-1} a_2}{ a_3} \eta^{-m}\quad \forall 0 < \rho < \rho_0.
 \end{equation}
Let 
\begin{equation}\label{defn-rho0}
\overline{\rho}_0 =\min \left\{ \rho_2, \frac{|a_2|\eta^{1-m}}{8|a_1| +1} \right\}. 
\end{equation} 
Suppose the left hand side inequality of \eqref{eq:one point singular:32} does not hold. Then there exists a constant $\rho_3\in (0,\overline{\rho}_{0})$ such that 
\begin{equation*}
q(\rho_3)< \frac{2^{1+m} a_2}{ a_3} \eta^{-m}.
\end{equation*}
By continuity of $q(\rho)$ on $(0,\infty)$, there exists a maximal interval $(\rho_4,\rho_3) , \; 0\le \rho_4<\rho_3$ such that  
\begin{equation}
\label{eq:one point singular:44}
q(\rho)< \frac{2^{1+m} a_2}{ a_3} \eta^{-m}\quad \forall \rho \in (\rho_4,\rho_3).
\end{equation} 
By \Cref{eq:one point singular:31}, \Cref{eq:one point singular:42}, \eqref{defn-rho0}, \Cref{eq:one point singular:44} and Lemma \ref{lemma:one point singular:1} we get
\begin{align}\label{eq:one point singular:43} 
q_{\rho}(\rho)
\le&\frac{q(\rho)\overline{w}(\rho)^{1-m}}{\rho^2} \left(a_2- \frac{a_3\ol{w}(\rho)^m}{2}q(\rho) \right) -\frac{a_{3}\ol{w}(\rho)}{2\rho^2}q(\rho)^2\notag\\
\le&\frac{q(\rho)\overline{w}(\rho)^{1-m}}{\rho^2} \left(a_2- \frac{a_3}{2}\cdot\left(\frac{\eta}{2}\right)^2\cdot\frac{2^{1+m} a_2}{ a_3} \eta^{-m} \right) -\frac{a_{3}\eta }{4\rho^2}q(\rho)^2\notag\\
=&-\frac{a_{3}\eta }{4\rho^2}q(\rho)^2
<0\qquad\qquad\quad\forall\rho\in(\rho_4,\rho_3).
\end{align}
Dividing \eqref{eq:one point singular:43}  by $q(\rho)^2$  and  integrating over $(\rho,\rho_3)$, 
$\rho_4<\rho<\rho_3$,
 \begin{align}\label{w-bar-ineqn2}
\overline{w}_{\rho}(\rho)\le \overline{w}_{\rho}(\rho_3) + \frac{ a_3 \eta}{4} \left( \frac{1}{\rho_3} - \frac{1}{\rho} \right)\quad\forall\rho_4<\rho<\rho_3.
 \end{align}
Integrating \eqref{w-bar-ineqn2} over $(\rho,\rho_3)$, $\rho_4<\rho<\rho_3$,
 \begin{align}\label{eq:one point singular:46} 
\overline{w}(\rho) \ge \overline{w}(\rho_3) - \overline{w}_{\rho}(\rho_3)(\rho_3-\rho) +  \frac{ a_3 \eta}{4}\log \left( \frac{\rho_3}{\rho} \right) -\frac{ a_3 \eta(\rho_3-\rho)}{4\rho_3}\quad\forall\rho_4<\rho<\rho_3.
 \end{align}
If $\rho_4 = 0$, then by letting $\rho\to 0 $ in  \Cref{eq:one point singular:46} we get 
\begin{equation}\label{w-bar-origin=0}
\lim_{\rho \to 0^+} \overline{w}(\rho) = \infty
\end{equation}
 which contradicts \Cref{w-bar-at-origin}.
Hence $\rho_4 > 0$ and  
\begin{equation}
\label{eq:one point singular:45}
q(\rho_4)=\frac{2^{1+m} a_2}{ a_3} \eta^{-m}.
\end{equation} 
Also by continuity of $q_{\rho}$ and $q$, \eqref{eq:one point singular:43} holds for $\rho=\rho_4$. Hence
\begin{equation}\label{q-derivative-ineqn1} 
q_{\rho}(\rho_4)\le-\frac{a_{3}\eta }{4\rho^2}q(\rho_4)^2<0
\end{equation}
By \eqref{eq:one point singular:45} and \eqref{q-derivative-ineqn1} there exists a constant $\delta_1\in (0,\rho_4)$ such that
\begin{equation}\label{q-lower-bd-inqen}
q(\rho)>\frac{2^{1+m} a_2}{ a_3} \eta^{-m}\quad\forall \rho_4-\delta_1\le\rho<\rho_4.
\end{equation}
If there exists $\rho_5\in (0,\rho_4-\delta_1)$ such that
\begin{equation}\label{q-lower-bd-inqen2}
q(\rho)<\frac{2^{1+m}a_2}{ a_3} \eta^{-m},
\end{equation}
then there exists a constant $\rho_6\in (\rho_5,\rho_4-\delta_1)$ such that
\begin{align}\label{q-lower-bd-inqen3}
&q(\rho)<\frac{2^{1+m} a_2}{ a_3} \eta^{-m}\quad\forall \rho_5\le\rho<\rho_6\quad\mbox{ and }\quad q(\rho_6)=\frac{2 a_2}{ a_3} \eta^{-m}\notag\\
\Rightarrow\quad&q_{\rho}(\rho_6)\ge 0.
\end{align}
On the other hand by the same argument as the proof of \eqref{eq:one point singular:43} we get
\begin{equation*} 
q_{\rho}(\rho_6)\le-\frac{a_{3}\eta }{4\rho^2}q(\rho_6)^2<0
\end{equation*}
which contradicts \eqref{q-lower-bd-inqen3}. Hence no such constant $\rho_5$ exists and
\begin{equation}\label{q-lower-bd-inqen4}
q(\rho)>\frac{2^{1+m} a_2}{ a_3} \eta^{-m}\quad\forall 0<\rho<\rho_4.
\end{equation}
We will now proof the right hand side inequality of \eqref{eq:one point singular:32}.
Without loss of generality we may assume that there exists a constant $\rho_7\in (0,\ol{\rho}_0)$ such that
\begin{equation}\label{q-lower-bd10}
q(\rho_7)>\frac{2^{m-1}a_2}{a_3} \eta^{-m}.
\end{equation} 
By \eqref{eq:one point singular:42} and \eqref{q-lower-bd10},
\begin{equation}\label{w-bar-q-lower-bd}
\ol{w}(\rho_7)q(\rho_7)>\frac{2^{m-2}a_2}{a_3} \eta^{1-m}.
\end{equation}
Then there exists a maximal interval $(\rho_8,\rho_7)$, $0\le\rho_8<\rho_7$,  such that 
\begin{equation}\label{q-lower-bd5}
\ol{w}(\rho)q(\rho)>\frac{2^{m-2}a_2}{a_3}\eta^{1-m}\quad\forall\rho_8<\rho<\rho_7.
\end{equation} 
By \Cref{eq:one point singular:31}, \eqref{eq:one point singular:42}, \eqref{defn-rho0}, \eqref{q-lower-bd5} and Lemma \ref {lemma:one point singular:1}
 \begin{align}\label{eq:one point singular:49} 
 (\ol{w}q)_{\rho}(\rho)=&m + \frac{\ol{w}(\rho)q(\rho)}{\rho} \left[ \left(  a_1 + \frac{ a_2 \overline{w}(\rho)^{1-m} }{ 4 \rho}\right) + \frac{3}{4\rho} (a_2 \overline{w}(\rho)^{1-m} -2a_3 \overline{w}(\rho)q(\rho)  ) + \frac{a_3}{2\rho} \overline{w}(\rho)q(\rho) \right]\notag\\
\ge&\frac{\ol{w}(\rho)\ol{q}(\rho)}{\rho} \left[  \left(a_1 + \frac{a_2}{4\rho}(\eta/2)^{1-m}\right)  +  \frac{3}{4\rho}   \left( a_2(\eta/2)^{1-m} - 2 a_3 2^{m-2} \frac{a_2}{a_3}\eta^{1-m}  \right)\right.\notag\\
&\qquad +\left. \frac{a_3}{2\rho}\ol{w}(\rho)q(\rho)   \right] \nonumber\\
\ge&\frac{a_3}{2\rho^2}\ol{w}(\rho)^2q(\rho)^2 >0\quad\forall\rho_8<\rho<\rho_7.
 \end{align}
Dividing \eqref{eq:one point singular:49} by $w(\rho)^2q(\rho)^2$ and integrating over $(\rho,\rho_7)$, $\rho_8<\rho<\rho_7$, we get
\begin{equation}\label{w-bar-w-ineqn}
\frac{\ol{w}_{\rho}(\rho)}{\ol{w}(\rho)}=\frac{1}{\ol{w}(\rho)q(\rho)}\ge\frac{1}{\ol{w}(\rho_7)q(\rho_7)}+\frac{a_3}{2}\left(\frac{1}{\rho}-\frac{1}{\rho_7}\right)\quad\forall \rho_8<\rho<\rho_7.
\end{equation}
Integrating \eqref{w-bar-w-ineqn} over $(\rho,\rho_7)$, $\rho_8<\rho<\rho_7$, we get
 \begin{align}\label{eq:one point singular:50} 
\log \ol{w}(\rho_7) - \log \ol{w}(\rho) \ge \frac{1}{\ol{w}(\rho_7) q(\rho_7)} (\rho_7-\rho) + \frac{a_3}{2}\log \left( \frac{\rho_7}{\rho} \right) -\frac{a_3}{2 \rho_7} \left(  \rho_7-\rho \right)\quad\forall \rho_8<\rho<\rho_7.
 \end{align}
 If $\rho_8=0$, then by letting $\rho\to 0$ in \eqref{eq:one point singular:50} we get
\eqref{w-bar-origin=0} which contradicts \eqref{w-bar-at-origin}.
Hence $\rho_8>0$ and
\begin{equation}\label{q-rho8-value}
\ol{w}(\rho_8)q(\rho_8)=\frac{2^{m-2}a_2}{a_3}\eta^{1-m}.
\end{equation}
By \eqref{q-rho8-value} and the same argument as the proof of \eqref{eq:one point singular:49} we get
\begin{equation*}
 (\ol{w}q)_{\rho}(\rho_8)\ge\frac{a_3}{2\rho_8^2}\ol{w}(\rho_8)^2q(\rho_8)^2 >0.
 \end{equation*}
Hence there exists a constant $0<\delta_1<\rho_8$ such that
\begin{equation}\label{bar-w-q-upper-bd1}
\ol{w}(\rho_8)q(\rho)<\frac{2^{m-2}a_2}{a_3}\eta^{1-m}\quad\forall\rho_8-\delta_1\le\rho<\rho_8.
\end{equation}
Suppose there exists a constant $0<\rho_9<\rho_8-\delta_1$ such that
\begin{equation*}
\ol{w}(\rho_9)q(\rho_9)>\frac{2^{m-2}a_2}{a_3} \eta^{1-m}.
\end{equation*}
Then there exists a constant $\rho_{10}\in (\rho_9,\rho_8-\delta_1)$ such that
\begin{align}
&\ol{w}(\rho)q(\rho)>\frac{2^{m-2}a_2}{a_3}\eta^{1-m}\quad\forall\rho_9<\rho<\rho_{10}\quad\mbox{ and }\quad
\ol{w}(\rho_{10})q(\rho_{10})=\frac{2^{m-2}a_2}{a_3}\eta^{1-m}\label{w-bar-q-lower-bd12}\\
\Rightarrow\quad&(\ol{w}q)_{\rho}(\rho_{10})\le 0.\label{w-bar-q-lower-bd13}
\end{align} 
On the other hand by \eqref{w-bar-q-lower-bd12} and the same argument as the proof of \eqref{eq:one point singular:49} ,
 \begin{equation*}
(\ol{w}q)_{\rho}(\rho_{10})\ge\frac{a_3}{2\rho_{10}^2}\ol{w}(\rho_{10})^2q(\rho_{10})^2 >0
 \end{equation*}
which contradicts \eqref{w-bar-q-lower-bd13}. Hence no such constant $\rho_9$ exists and
\begin{equation}\label{w-bar-q-lower-bd20}
\ol{w}(\rho)q(\rho)\le\frac{2^{m-2}a_2}{a_3} \eta^{1-m}\quad\forall 0<\rho<\rho_8.
\end{equation}
By \eqref{eq:one point singular:42} and \eqref{w-bar-q-lower-bd20},
\begin{equation}\label{w-bar-q-lower-bd22}
q(\rho)\le\frac{2^{m-1}a_2}{a_3} \eta^{-m}\quad\forall 0<\rho<\rho_8.
\end{equation}
Let $\rho_0=\min (\rho_4,\rho_8)$. Then by \eqref{q-lower-bd-inqen4} and \eqref{w-bar-q-lower-bd22}, \eqref{eq:one point singular:32} holds. Then by \eqref{eq:one point singular:32} and an argument similar to the proof of [Lemma 2.2, \cite{HK19singul}], which is essentially L'Hospital rule,
 we get
  \begin{equation*}
  \lim_{\rho \to 0^+} q(\rho) = \frac{a_2}{a_3} \eta^{-m}.
  \end{equation*}
and \eqref{derivative of barw at 0} follows. Finally by \eqref{eq:one point singular:27}, \eqref{derivative of barw at 0} and the same argument as the proof of Lemma 2.2 of \cite{HK19singul} we get \eqref{f-derivative-at-origin} and the lemma follows.
\end{proof}  

By \eqref{w-bar-at-origin}, Lemma \ref{w-rho-derivative-limit-at-origin-lem} and an argument similar to the proof of Lemma 2.3 of \cite{HK19singul} we have the following result.

\begin{lemma}\label{w-rho-2nd-derivative-limit-at-origin-lem}
Let $n \ge 3$, $0<m < (n-2)/n$, $\eta>0$ and $\alpha$, $\alpha'$, $\beta$, $\beta'$, $\rho_1$, satisfy  \eqref{alpha'-beta'-defn} and \eqref{eq:alpha beta relation3}. Then 
     \begin{align*}
    \lim_{\rho \to 0^+} \overline{w}_{\rho \rho } (\rho) = \frac{a_3(ma_3-a_1)}{a_2^2} \eta^{2m-1} 
    \end{align*}
    where $a_1,\ a_2 $ and $a_3$  are constants given by \Cref{eq:one point singular:30}. Hence $\overline{w}$ can be extended to a function in  $C^2([0,\infty))$ by defining $\overline{w}(0)$, $\overline{w}_{\rho}(0)$ and $\overline{w}_{\rho\rho}(0)$ by \eqref{eq:one point singular:28}. 
  \end{lemma}
  
By Lemma \ref{w-rho-derivative-limit-at-origin-lem}, Lemma \ref{w-rho-2nd-derivative-limit-at-origin-lem} and Taylor's expansion for $\overline{w}$ and $\overline{w}_{\rho}$ near the origin (cf.  \cite{HK19singul}),  Theorem \ref{asymptotic expansion of self-similar solution at the origin thm} follows.

\section{Asymptotic large time behavior of singular solutions} \label{sec:asympt-large-time}

In this section we will prove the asymptotic large time behaviour of singular solutions of \eqref{eq-fde2-global-except-0}. We will first prove Theorem \ref{weighted contraction theorem} which is a weighted $L^1$-contraction with weight $\varphi_{\mu}$. 
   
\noindent{\it Proof of Theorem \ref{weighted contraction theorem}:}

\noindent We first observe that by an argument similar to the proof of Theorem 2.2 of \cite{TY22infini}, we have
\begin{align}\label{eq:one point singular:55} 
&\int_{\mathbb{R}^n\setminus\{0\}}|u-v| (x,t) \psi(x) \,\mathrm{d} x 
-\int_{\mathbb{R}^n\setminus\{0\}}|u_0-v_0| (x) \psi(x) \,\mathrm{d} x \notag\\
\le&\int_0^t\int_{\mathbb{R}^n\setminus\{0\}}|u^m-v^m|(x,s)\Delta\psi(x)\,\mathrm{d} x\,ds \quad\forall\psi\in C_0^{\infty}(\mathbb{R}^n),t>0.
\end{align}
We will now use \eqref{eq:one point singular:55} and a modification of the proof of Theorem 1.2 of \cite{HK17asympt} to prove the theorem. Since the proof is similar to the proof of Theorem 1.2 of \cite{HK17asympt}, we will only sketch the argument here. We choose $\xi \in C^{\infty}_0(\mathbb{R}^n )$ such that $0 \le \xi <1 , \;  \xi=1$ for $|x|\le 1 $, and $\xi=0$ for $|x| >2$. For any $R>2$, we let $\xi_R(x):= \xi(x/R)$. Putting $\psi(x)=\xi_R (x) \varphi_{\mu}(x)$ in \Cref{eq:one point singular:55}, we have
\begin{align}\label{u-v-difference-integral-ineqn}
&\int_{\mathbb{R}^n }|u-v| (x,t) \xi_R(x) \varphi_{\mu}(x) \,\mathrm{d} x 
-\int_{\mathbb{R}^n }|u_0-v_0| (x) \xi_R(x) \varphi_{\mu}(x) \,\mathrm{d} x \notag\\
\le&\int_0^t\int_{\mathbb{R}^n  }\frac{ |u^m-v^m|(x,s) }{m} (x,t) \Delta \left(  \xi_R(x) \varphi_{\mu}(x) \right) \,\mathrm{d} x \,ds\nonumber\\
=&\int_0^t\int_{\mathbb{R}^n } \frac{|u^m-v^m|(x,s) }{m} \{\varphi_{\mu}(x) \Delta \xi_R(x)+ 2\nabla \varphi_{\mu}(x) \cdot \nabla \xi_R(x) +\xi_R(x)\Delta \varphi_{\mu}(x)     \} \,\mathrm{d} x\,ds.
\end{align}

Then by \eqref{laplace-phi-negative}, \eqref{phi-asymptotic behaviour2},  \eqref{phi-derivative-asymptotic behaviour2}, \eqref{solns-lower-upper-bd} and \eqref{u-v-difference-integral-ineqn},
 \begin{align}\label{eq:one point singular:57} 
&\int_{\mathbb{R}^n }|u-v| (x,t) \xi_R(x) \varphi_{\mu}(x) \,\mathrm{d} x 
-\int_{\mathbb{R}^n }|u_0-v_0| (x) \xi_R(x) \varphi_{\mu}(x) \,\mathrm{d} x \notag\\
\le& C R^{-2-\mu} \int_0^t\int_{B_{2R}\setminus B_R} |u^m-v^m|(x,s)\,\mathrm{d} x\,ds, \quad\forall R>R_0\notag\\ 
\le&C R^{-2-\mu} \int_0^t\int_{B_{2R}\setminus B_R} V_{\lambda_2}(x,s)^m  \,\mathrm{d} x\,ds, \quad R>R_0. 
  \end{align}
 By an argument similar to the proof of Theorem 1.2 of \cite{HK17asympt},
\begin{align}\label{eq:one point singular:59} 
\int_0^t\int_{B_{2R}\setminus B_R}V_{\lambda_2}(x,s)^m  \,\mathrm{d} x \,\mathrm{d}s
\le&C\left\{R^{n+2 - \frac{\alpha}{\beta}} +R^2\log (tR^{-1/\beta}) + R^2t^{n\beta -\alpha} \right\} \quad 
\mbox{ if } R >t^{\beta}>0.
\end{align}
By \cref{eq:one point singular:57} and \eqref{eq:one point singular:59},
  \begin{align*}
 & \int_{\mathbb{R}^n } |u-v| (x,t) \xi_R(x) \varphi_{\mu}(x) \,\mathrm{d}x-\int_{\mathbb{R}^n } |u_0-v_0| (x) \xi_R(x) \varphi_{\mu}(x) \,\mathrm{d}x   \nonumber\\
\le& C_t \left(  R^{n-\frac{\alpha}{\beta} - \mu}  + R^{-\mu} \log R + R^{-\mu} \right)\quad\forall R>\max(R_0,t^{\beta})\\
\to&0\quad\mbox{ as }R \to \infty.
 \end{align*}
for some constant $C_t>0$ and \eqref{eq:one point singular:52} follows. By a similar argument we also get \eqref{eq:one point singular:54}
and the theorem follows.

{\hfill$\square$\vspace{6pt}}

We next recall Theorem 1.3 of \cite{HK17asympt}.
   
\begin{theorem}(Theorem 1.3 of \cite{HK17asympt})\label{existence thm6}
Let $n \ge 3,$ $0< m< \frac{n-2}{m}$, and $\frac{2}{1-m} < \gamma < \frac{n-2}{m}$. Let $u_0$ satisfies \Cref{eq-fde2-initial} for some constants $A_2>A_1>0$. Then there exists a unique solution $u$ of \eqref{eq-fde2-global-except-0}  satisfying \eqref{eq:one point singular:62}
where $V_{\lambda_i}$ for $i=1,2$, are given by \Cref{eq:V lambda} with $\alpha , \;  \beta$ satisfying \Cref{eq:alpha beta relation}, \eqref{eq:gamma slpha beta relation6}, and $\lambda_i = A_i ^{(1-m)\beta} $ for $i=1,2$, respectively. Moreover
\begin{equation*}
u_t \le \frac{u}{ (1-m)t}\quad \text{ in } (\mathbb{R}^n  \setminus  \{    0     \}) \times (0,\infty).
\end{equation*}
\end{theorem}

We will now study the asymptotic large time behaviour of the singular solution $u$ of \eqref{eq-fde2-global-except-0} with initial value $u_0$ satisfying \Cref{eq-fde2-initial} for some constants $A_2>A_1>0$ and $n \le \gamma < \frac{ n-2}{m}$.  We first observe that the rescaled function $\widetilde{u}$ of the solution $u$  given by \eqref{eq:u tilde} satisfies 
\begin{equation}\label{eq:one point singular:65}
\widetilde{u}_{\tau}=\Delta(\widetilde{u} ^m /m) + \alpha \widetilde{u} + \beta y \cdot  \nabla \widetilde{u}\quad\mbox{ in }(\mathbb{R}^n\setminus\{0\})\times (-\infty,\infty).
\end{equation}
In particular the rescaled function $\widetilde{V}_{\lambda}$ for the function $V_{\lambda}$ given by \eqref{eq:V lambda} is equal to $f_{\lambda}(y)$.
We will now assume $n \le \gamma < \frac{ n-2}{m}$ for the rest of the paper. For results for the complimentary interval $\frac{2}{1-m} < \gamma < n$, we refer the readers to the paper \cite{HK17asympt} by K.M.~Hui and Soojung Kim. 

\begin{lemma}[Strong contraction principle for the rescaled functions]
\label{Strong contraction principle for the rescaled functions lemma}
Let $n\ge 3$, $0< m < \frac{n-2}{n}$,  $0<\mu<n-2$, $\beta < 0 , \;  \alpha=\frac{2\beta-1}{1-m}$ and let $\gamma$ satisfy \eqref{eq:gamma slpha beta relation6} and \eqref{gamma-range1}. 
Let $\widetilde{u}$,  $\widetilde{v}$, be solutions of \Cref{eq:one point singular:65} in $ (\mathbb{R}^n\setminus\{0\}) \times (0,\infty)$ with initial value $\widetilde{u}_0$,  $\widetilde{v}_0$, respectively such that
  \begin{equation}\label{eq:one point singular:64}
  f_{\lambda_d1} \le \widetilde{u},\widetilde{v} \le f_{\lambda_2} \quad \text{ in } (\mathbb{R}^n\setminus\{0\}) \times (0,\infty)
  \end{equation}
  for some constants $\lambda_1>\lambda_2>0$. Suppose that 
  \begin{equation}\label{bar-u-v-intitial-value-L1-norm}
    0 \not\equiv \widetilde{u}_0-\widetilde{v}_0 \in L^1 (\varphi_{\mu}; \mathbb{R}^n\setminus\{0\} ).
  \end{equation}
for some function  $\varphi_{\mu}$ given by \Cref{eq:varphi}.
 Then 
 \begin{equation*}
 \left \|    \widetilde{u}(\cdot,\tau) - \widetilde{v}(\cdot,\tau)    \right \|_{L^1 (\varphi_{\mu}; \mathbb{R}^n\setminus\{0\} )} < \left \|    \widetilde{u}_0(\cdot) - \widetilde{v}_{0}(\cdot)    \right \|_{L^1 (\varphi_{\mu}; \mathbb{R}^n\setminus\{0\} )} \quad \forall \tau>0.
 \end{equation*}
\end{lemma}
\begin{proof}
We will use a modification of the proof of Lemma 4.2 of \cite{HK17asympt} to prove the lemma. 
Let 
\begin{equation}\label{u-bar-u-v-defn}
\left\{\begin{aligned}
&u(x,t)=t^{-\alpha}\4{u}(t^{-\beta}x,\tau)\quad\forall x\in\mathbb{R}^n\setminus\{0\},t=e^{\tau}\ge 1\\
&v(x,t)=t^{-\alpha}\4{v}(t^{-\beta}x,\tau)\quad\forall x\in\mathbb{R}^n\setminus\{0\},t=e^{\tau}\ge 1.
\end{aligned}\right.
\end{equation}
Then 
\begin{equation}\label{u-v-L1-norm-t=1}
u(x,1)=\4{u}_0(x),\quad v(x,1)=\4{v}_0(x), 
\end{equation}
and $u$, $v$ are solutions of \eqref{fde} in $(\mathbb{R}^n\setminus\{0\})\times (1,\infty)$ with initial 
values $\4{u}_0$,$\4{v}_0$, at time $t=1$. By \eqref{bar-u-v-intitial-value-L1-norm}, \eqref{u-bar-u-v-defn}, \eqref{u-v-L1-norm-t=1} and an argument similar to  the proof of Theorem 2.2 of \cite{TY22infini}, we have 
\begin{equation*}
u-v \in L^1(K \times (1,t))
\end{equation*}
for any $t>1$ and any compact subset $K\subset\mathbb{R}^n$. Hence 
\begin{equation}\label{loc-L1-bd1}
\widetilde{u}-\widetilde{v} \in L^1(K\times (0,\tau))
\end{equation}
for any $\tau>0$ and  any compact subset $K\subset\mathbb{R}^n$.       
Let $q=|\widetilde{u}-\widetilde{v}|$. Since there exists a constant $C>0$ such that 
\begin{equation*}
|a^m-b^m|\le C|a-b|^m\quad\forall a,b\in\mathbb{R}^+,
\end{equation*}
we have
\begin{equation}\label{loc-L1-bd2}
\int_0^{\tau}\int_K|\widetilde{u}^m-\widetilde{v}^m|\,dy\,ds\le C(|K|\tau)^{1-m}\left(\int_0^{\tau}\int_K|\widetilde{u}-\widetilde{v}|\,dy\,ds\right)^m<\infty
\end{equation}
for any $\tau>0$ and  any compact subset $K\subset\mathbb{R}^n$.
By \eqref{eq:one point singular:65} and the Kato inequality (\cite{DK07degene}, \cite{Kat13pertur}, \cite{HK17asympt}),
\begin{equation}\label{q-ineqn}
q_{\tau} \le \Delta (\widetilde{a}q/m) + \beta\mbox{div}(yq) + (\alpha-n\beta) q \quad \text{ in } \mathcal{D}'( (\mathbb{R}^n\setminus\{0\})  \times (0,\infty)).  
\end{equation}
where 
\begin{equation*}
mf_{\lambda_2}^{m-1}(y) \le \widetilde{a}(y,\tau) := \int_0^1 \frac{m \,\mathrm{d}s}{\{s\widetilde{u} +(1-s)\widetilde{v}\}^{1-m}}\le m f_{\lambda_1}^{m-1}(y).
\end{equation*}
We now let $\psi \in C_c^{\infty} (\mathbb{R}^n )$, $\ol{\phi}\in C^{\infty}(\mathbb{R}^n )$, $0\le\ol{\phi}\le 1$, such that $\ol{\phi}(x)=0\quad\forall |x|\le 1/2$ and $\ol{\phi}(x)=1\quad\forall |x|\ge 1$. Let $k>2/(1-m)$ be an integer and $\phi=\ol{\phi}^k$. For any $0<\3<1$, let $\phi_{\3}(x)=\phi(x/\3)$ and $\ol{\phi}_{\3}(x)=\ol{\phi}(x/\3)$. Then $\phi_{\3}=(\ol{\phi}_{\3})^k$, $\psi\phi_{\3}\in C_c^{\infty}(\mathbb{R}^n\setminus\{0\})$ for any $0<\3<1$ and supp$\,\psi\subset B_{R_1}$ for some constant $R_1>0$. 

Multiplying \eqref{q-ineqn} by $\psi\phi_{\3}$, $0<\3<1$, and integrating over $(\mathbb{R}^n\setminus\{0\})\times (0,\tau)$, $\tau>0$, we get 
\begin{align}\label{q-integral-ineqn}
&\int_{\mathbb{R}^n\setminus\{0\}}|\widetilde{u}-\widetilde{v}|(y,\tau)\psi(y)\phi_{\3}(y)\,dy
-\int_{\mathbb{R}^n\setminus\{0\}}|\widetilde{u}_0-\widetilde{v}_0|(y)\psi(y)\phi_{\3}(y)\,dy
\notag\\
\le&\frac{1}{m}\int_0^{\tau}\int_{\mathbb{R}^n\setminus\{0\}}|\widetilde{u}^m-\widetilde{v}^m|\Delta(\psi\phi_{\3})\,dy\,ds  
-\beta\int_0^{\tau}\int_{\mathbb{R}^n }y\cdot\nabla (\psi\phi_{\3}) |\widetilde{u}-\widetilde{v}|(y,s)\,dy\,ds\notag\\
&\qquad + (\alpha-n\beta) \int_0^{\tau}\int_{\mathbb{R}^n\setminus\{0\}}\psi\phi_{\3}|\widetilde{u}-\widetilde{v}| \,dy\,dt \nonumber\\
=&\frac{1}{m}\int_0^{\tau}\int_{\mathbb{R}^n\setminus\{0\}}\phi_{\varepsilon}\Delta \psi | \widetilde{u}^m-\widetilde{v}^m|\,dy\,ds
+\int_0^{\tau}\int_{\mathbb{R}^n\setminus\{0\}}\phi_{\varepsilon}\left(-\beta y\cdot\nabla\psi+  (\alpha-n\beta)\psi\right)|\widetilde{u}-\widetilde{v}|\,dy\,ds\notag\\
&\quad +\frac{1}{m}\int_0^{\tau}\int_{\mathbb{R}^n\setminus\{0\}}\left(2 \nabla \psi \cdot \nabla \phi_{\varepsilon} +\psi\Delta \phi_{\varepsilon} \right) | \widetilde{u}^m-\widetilde{v}^m|\,dy\,ds
 -\beta \int_0^{\tau}\int_{\mathbb{R}^n\setminus\{0\}}\psi (y\cdot\nabla \phi_{\varepsilon})|\widetilde{u}-\widetilde{v}| \,dy\,ds\notag\\
=&I_1+I_2+I_3+I_4.
\end{align} 
By \eqref{loc-L1-bd1} and \eqref{loc-L1-bd2} we have 
\begin{equation}\label{I1-limit}
I_1\to \frac{1}{m}\int_0^{\tau}\int_{\mathbb{R}^n\setminus\{0\}} | \widetilde{u}^m-\widetilde{v}^m|\Delta \psi\,dy\,ds\quad\mbox{ as }\3\to 0,
\end{equation}
\begin{equation}\label{I2-limit}
I_2\to \int_0^{\tau}\int_{\mathbb{R}^n\setminus\{0\}}\left(-\beta y\cdot\nabla\psi+  (\alpha-n\beta)\psi\right)|\widetilde{u}-\widetilde{v}|\,dy\,ds\quad\mbox{ as }\3\to 0,
\end{equation}
\begin{equation}\label{I4-limit}
|I_4|\le C\int_0^{\tau}\int_{\3/2\le|y|\le\3}|\widetilde{u}-\widetilde{v}| \,dy\,ds\to 0\quad\mbox{ as }\3\to 0
\end{equation}
and
\begin{align}\label{I3-limit}
|I_3|\le&C \int_0^{\tau}\int_{\3/2\le |y|\le\3}\left(|\nabla \phi_{\varepsilon}| \phi_{\varepsilon}^{-m} +|\Delta \phi_{\varepsilon}|\phi_{\varepsilon}^{-m}\right) (| \widetilde{u}-\widetilde{v}|\phi_{\varepsilon})^m\,dy\,ds\notag\\
\le&C \left\{\left(\int_{\3/2\le |y|\le\3}\left(|\nabla \phi_{\varepsilon}| \phi_{\varepsilon}^{-m}\right)^{\frac{1}{1-m}}\,dy\right)^{1-m}
+\left(\int_{\3/2\le |y|\le\3}\left(|\Delta\phi_{\varepsilon}| \phi_{\varepsilon}^{-m}\right)^{\frac{1}{1-m}}\,dy\right)^{1-m}\right\}\cdot\notag\\
&\quad\cdot\int_0^{\tau}\left(\int_{\3/2\le |y|\le\3}| \widetilde{u}-\widetilde{v}|\phi_{\varepsilon}\,dy\right)^m\,ds\notag\\
\le&C\tau^{1-m} \left\{\left(\int_{\3/2\le |y|\le\3}\left(|\nabla \phi_{\varepsilon}| \phi_{\varepsilon}^{-m}\right)^{\frac{1}{1-m}}\,dy\right)^{1-m}
+\left(\int_{\3/2\le |y|\le\3}\left(|\Delta\phi_{\varepsilon}| \phi_{\varepsilon}^{-m}\right)^{\frac{1}{1-m}}\,dy\right)^{1-m}\right\}\cdot\notag\\
&\quad\cdot\left(\int_0^{\tau}\int_{\3/2\le |y|\le\3}| \widetilde{u}-\widetilde{v}|\phi_{\varepsilon}\,dy\,ds\right)^m.
\end{align}
Since $k>\frac{2}{1-m}$, we have
\begin{equation*}
|\nabla \phi_{\varepsilon}| \phi_{\varepsilon}^{-m}=|\nabla\ol{\phi}_{\varepsilon}^k|\ol{\phi}_{\varepsilon}^{-km}=k\ol{\phi}_{\varepsilon}^{(1-m)k-1}|\nabla\ol{\phi}_{\varepsilon}|\le C|\nabla\ol{\phi}_{\varepsilon}|\le C\3^{-1}
\end{equation*}
and
\begin{equation*}
|\Delta \phi_{\varepsilon}| \phi_{\varepsilon}^{-m}
=|\Delta \ol{\phi}_{\varepsilon}^k|\ol{\phi}_{\varepsilon}^{-km}
=|k(k-1)\ol{\phi}_{\varepsilon}^{(1-m)k-2}|\nabla\ol{\phi}_{\varepsilon}|^2+\ol{\phi}_{\varepsilon}^{(1-m)k-1}|\Delta \ol{\phi}_{\varepsilon}|
\le C(|\nabla\ol{\phi}_{\varepsilon}|^2+|\Delta\ol{\phi}_{\varepsilon}|)\le C\3^{-2}.
\end{equation*}
Hence
\begin{align}\label{phi-integral-estimate1}
&\left(\int_{\3/2\le |y|\le\3}\left(|\nabla \phi_{\varepsilon}| \phi_{\varepsilon}^{-m}\right)^{\frac{1}{1-m}}\,dy\right)^{1-m}
+\left(\int_{\3/2\le |y|\le\3}\left(|\Delta\phi_{\varepsilon}| \phi_{\varepsilon}^{-m}\right)^{\frac{1}{1-m}}\,dy\right)^{1-m}\notag\\
\le &C\left(\int_{\3/2\le |y|\le\3}\3^{-\frac{2}{1-m}}\,dy\right)^{1-m}\notag\\
\le &C\3^{n(1-m)-2}\to 0\qquad\quad\mbox{ as }\3\to 0.
\end{align}
 By \eqref{I3-limit} and \eqref{phi-integral-estimate1},
\begin{equation}\label{I3-limit3}
I_3\to 0\quad\mbox{ as }\3\to 0.
\end{equation}
Letting $\3\to 0$ in  \eqref{q-integral-ineqn} by \eqref{I1-limit}, \eqref{I2-limit}, \eqref{I4-limit} and \eqref{I3-limit3}, we get
\begin{align}\label{q-integral-ineqn2}
&\int_{\mathbb{R}^n\setminus\{0\}}|\widetilde{u}-\widetilde{v}|(y,\tau)\psi(y)\,dy
-\int_{\mathbb{R}^n\setminus\{0\}}|\widetilde{u}_0-\widetilde{v}_0|(y)\psi(y)\,dy
\notag\\
\le&\frac{1}{m}\int_0^{\tau}\int_{\mathbb{R}^n\setminus\{0\}} | \widetilde{u}^m-\widetilde{v}^m|\Delta \psi\,dy\,ds
-\beta\int_0^{\tau}\int_{\mathbb{R}^n\setminus\{0\}} y\cdot\nabla\psi|\widetilde{u}-\widetilde{v}|\,dy\,ds\notag\\
&\qquad +  (\alpha-n\beta)\int_0^{\tau}\int_{\mathbb{R}^n\setminus\{0\}}\psi|\widetilde{u}-\widetilde{v}|\,dy\,ds
\end{align} 
holds for any $\psi \in C_c^{\infty} (\mathbb{R}^n )$. We now choose $\xi\in C_c^{\infty}(\mathbb{R}^n)$, $0\le \xi\le 1$, such that $\xi(x)=1\quad\forall |x|\le 1$ and $\xi(x)=0$ for any $|x|\ge 2$. For any $R>2$ we let $\xi_R(x)=\xi(x/R)$. Putting $\psi=\xi_R\varphi_{\mu}$ in \eqref{q-integral-ineqn2} we have

\begin{align}\label{q-integral-ineqn3}
&\int_{\mathbb{R}^n\setminus\{0\}}|\widetilde{u}-\widetilde{v}|(y,\tau)\xi_R\varphi_{\mu}\,dy
-\int_{\mathbb{R}^n\setminus\{0\}} |\widetilde{u}_0-\widetilde{v}_0|(y)\xi_R\varphi_{\mu}\,dy
\notag\\
\le&\frac{1}{m}\int_0^{\tau}\int_{\mathbb{R}^n\setminus\{0\}} | \widetilde{u}^m-\widetilde{v}^m|\Delta (\xi_R\varphi_{\mu})\,dy\,ds
-\beta\int_0^{\tau}\int_{\mathbb{R}^n\setminus\{0\}}y\cdot\nabla(\xi_R\varphi_{\mu})|\widetilde{u}-\widetilde{v}|\,dy\,ds\notag\\
&\qquad +  (\alpha-n\beta)\int_0^{\tau}\int_{\mathbb{R}^n\setminus\{0\}}|\widetilde{u}-\widetilde{v}|\xi_R\varphi_{\mu}\,dy\,ds\notag\\
\le&\int_0^{\tau} \int_{\mathbb{R}^n }\left(\frac{\widetilde{a}}{m} \Delta \varphi_{\mu} - \beta y \cdot \nabla \varphi_{\mu} + (\alpha-n\beta) \varphi_{\mu}\right)\xi_R q \,\mathrm{d}y \,\mathrm{d}s \nonumber\\
\quad & +  \int_0^{\tau} \int_{B_{2R} \setminus B_R}\left(\frac{\widetilde{a}}{m} \varphi_{\mu}\Delta \xi_R + 2\frac{\widetilde{a}}{m}\nabla \xi_R \cdot \nabla \varphi_{\mu} - \beta (y \cdot\nabla \xi_R)\varphi _{\mu}\right)q \,\mathrm{d}y \,\mathrm{d}s.
\end{align} 
By \eqref{eq:gamma slpha beta relation6}, \eqref{gamma-range1}, \eqref{laplace-phi-negative}, \eqref{phi-bd} and \eqref{phi'-negative},
\begin{equation}\label{phi-expression-negative}
\frac{\widetilde{a}}{m} \Delta \varphi_{\mu} - \beta y \cdot \nabla \varphi_{\mu} + (\alpha-n\beta) \varphi_{\mu}<0 \quad \forall x \in \mathbb{R}^n.
\end{equation}
Since $f_{\lambda_2}$ satisfies \eqref{f-infty-behaviour} for some constant $D(\eta)>0$ with $\eta=\lambda_2^{\frac{2}{1-m}}$, there exists a constant $R_1>R_0$ and a constant $C>0$ such that
\begin{equation}\label{f-decay-infinity}
f_{\lambda_2}(r)\le Cr^{-\frac{n-2}{m}}\quad\forall r\ge R_1.
\end{equation}
By \eqref{phi-asymptotic behaviour2}, \eqref{phi-derivative-asymptotic behaviour2}, \eqref{eq:one point singular:64} and \eqref{f-decay-infinity},
\begin{align}\label{eq:one point singular:66}
& \left| \int_0^{\tau}\int_{B_{2R}\setminus B_R}\left(\frac{\widetilde{a}}{m}\varphi_{\mu}\Delta \xi_R
+ 2\frac{\widetilde{a}}{m}\nabla \xi_R \cdot \nabla \varphi_{\mu} - \beta (y \cdot\nabla \xi_R) \varphi_{\mu} \right)q \,\mathrm{d}y \,\mathrm{d}s   \right|  \nonumber\\
\le& C\left( R^{-2-\mu}\int_{B_{2R} \setminus B_R}  f^m_{\lambda_2} \,\mathrm{d}y + R^{-\mu}\int_{B_{2R} \setminus B_R}  f_{\lambda_2} \,\mathrm{d}y\right) \tau \quad\forall R>R_1,\tau>0\nonumber\\
\le& C\left( R^{-2-\mu}\int_{B_{2R} \setminus B_R}  |x|^{-n+2} \,\mathrm{d}y + R^{-\mu}\int_{B_{2R} \setminus B_R}  |x|^{\frac{-n+2}{m}} \,\mathrm{d}y\right) \tau\quad\forall R>R_1,\tau>0\nonumber\\
\le& C \left(  R^{-\mu}+ R ^{ n-\frac{n-2}{m} -\mu}   \right)\tau\quad\forall R>R_1,\tau>0\notag\\
\to& 0\quad\mbox{ as }R\to\infty.
\end{align}
Hence letting $R\to\infty$ in \eqref{q-integral-ineqn3}, by \eqref{phi-expression-negative} and \eqref{eq:one point singular:66} we have
\begin{align*}
&\int_{\mathbb{R}^n\setminus\{0\}}|\widetilde{u}-\widetilde{v}|(y,\tau)\phi_{\mu}\,dy
-\int_{\mathbb{R}^n\setminus\{0\}}|\widetilde{u}_0-\widetilde{v}_0|(y)\phi_{\mu}\,dy
\notag\\
\le&\lim_{R\to\infty}\int_0^{\tau} \int_{\mathbb{R}^n\setminus\{0\} }\left(\frac{\widetilde{a}}{m} \Delta \varphi_{\mu} - \beta y \cdot \nabla \varphi_{\mu} + (\alpha-n\beta) \varphi_{\mu}\right)\xi_R q \,\mathrm{d}y \,\mathrm{d}s \nonumber\\
<&0
\end{align*} 
and the lemma follows.

\end{proof}

By an argument similar to the proof of Lemma 4.3 and Lemma 4.4 of \cite{HK17asympt} but with Lemma 4.2 of \cite{HK17asympt} being replaced by Lemma \ref{Strong contraction principle for the rescaled functions lemma}
in the proof there we have the following results.

\begin{lemma}\label{compare principle two lemma}
 Let $n\ge 3$, $0< m < \frac{n-2}{n}$, $0<\mu<n-2$, $\beta < 0 , \;  \alpha=\frac{2\beta-1}{1-m}$ and let $\gamma$ satisfy \eqref{eq:gamma slpha beta relation6} and \eqref{gamma-range1}. 
Let $\widetilde{u}$,  $\widetilde{v}$, be solutions of \Cref{eq:one point singular:65}  with initial value $\widetilde{u}_0$,  $\widetilde{v}_0$, respectively such that \eqref{eq:one point singular:64}
  holds for some constants $\lambda_1>\lambda_2>0$. Suppose that there exists a constant $\lambda_0 \in [\lambda_2,\lambda_1]$ such that 
  \begin{equation*}
  \widetilde{u}_0 - f_{\lambda_0} \in L^1(\varphi_{\mu} ;\mathbb{R}^n\setminus\{0\} ) 
  \end{equation*}
  and
  \begin{equation*}
  \lim_{i \to \infty} || \widetilde{u}(\cdot, \tau_i) - \widetilde{v}_0 || _{L^1 (\varphi_{\mu}; \mathbb{R}^n\setminus\{0\}) } =0 
  \end{equation*}
for some sequence $ \{    \tau_i     \}^{\infty}_{i=1} $ such that $\tau_i \to \infty$ as $i \to \infty$. Then 
\begin{equation*}
|| \widetilde{v}_0 - f_{\lambda_0} || _{L^1 (\varphi_{\mu};\mathbb{R}^n\setminus\{0\}) } \le || \widetilde{u}_0 -f_{\lambda_0} || _{L^1 (\varphi_{\mu};\mathbb{R}^n\setminus\{0\} ) }
\end{equation*}
and
\begin{equation*}
||\widetilde{v}(\cdot , \tau_i) - f_{\lambda_0} || _{L^1 (\varphi_{\mu}; \mathbb{R}^n\setminus\{0\} ) } = || \widetilde{v}_0 - f_{\lambda_0} || _{L^1 (\varphi_{\mu};\mathbb{R}^n\setminus\{0\} ) } \quad \forall \tau >0.
\end{equation*}

\end{lemma}

\begin{lemma}\label{limit-lemma 2}
Let $n\ge 3$, $0< m < \frac{n-2}{n}$, $\beta < 0 , \;  \alpha=\frac{2\beta-1}{1-m}$ and let $\gamma$ satisfy \eqref{eq:gamma slpha beta relation6} and \eqref{gamma-range1}.  Let $u_0$ satisfy \Cref{eq-fde2-initial} and \eqref{u0-L1-bd5}
for some constants $A_2\ge A_0\ge A_1>0$ and $0< \mu < n-2$. Let $u$ be the solution of \Cref{fde} which satisfies \Cref{eq:one point singular:62} with $\lambda_i=A_i^{(1-m)\beta}$ for $i=1,2$,  and let $\widetilde{u}(y,\tau)$ be given by \Cref{eq:u tilde}. Let $ \{    \tau_i     \}_{i=1}^{\infty}$ be a sequence such that $\tau_i \to \infty$ as $i \to \infty$ and 
\begin{equation*}
\widetilde{u}_i(\cdot,\tau)  :=  \widetilde{u}(\cdot,\tau_i+ \tau) \quad \forall \tau \in \mathbb{R}.
\end{equation*} 
Then there exists a subsequence of $ \{     \widetilde{u}_i     \}_{i=1}^{\infty} $, which we still denote by $ \{    \widetilde{u}_i     \}_{i=1}^{\infty}$, and an enternal solution $\widetilde{v}$ of \Cref{eq:one point singular:65} in $(\mathbb{R}^n \setminus  \{    0     \} \times (-\infty,\infty)$ such that $\widetilde{u}_i$ converges to $\widetilde{v}$ uniformly on every compact subset of $(\mathbb{R}^n \setminus \{0 \} \times (-\infty,\infty)$ as $i \to \infty$. Moreover 
\begin{equation*}
\widetilde{u}(\cdot,0) - f_{\lambda_0} \in L^1(\varphi_{\mu} ; \mathbb{R}^n\setminus\{0\} )
\end{equation*}
where $\lambda_0 :=A_0^{(1-m)\beta}$ and 
\begin{equation*}
\lim_{i \to \infty} || \widetilde{u}_i(\cdot,\tau) - \widetilde{v}(\cdot,\tau)|| _{L^1(\varphi_{\mu};\mathbb{R}^n\setminus\{0\} )}=0 \quad \tau \in \mathbb{R}.
\end{equation*}
 
\end{lemma} 

\noindent{\it Proof of Theorem \ref{asymptotic large time behaviour of solution thm}:}

\noindent By the same argument as the proof of Theorem 1.4 of \cite{HK17asympt} but with Theorem 1.3, Lemma 4.2, Lemma 4.3 and Lemma 4.4 of \cite{HK17asympt} being replaced by Theorem \ref{existence thm6}, Lemma \ref{Strong contraction principle for the rescaled functions lemma}, Lemma \ref{compare principle two lemma} and Lemma \ref{limit-lemma 2} of this paper
the theorem follows.

{\hfill$\square$\vspace{6pt}}

\end{document}